\documentclass{amsart}
 \pdfoutput=1
 
\usepackage{amssymb}
\usepackage{amsmath}
\usepackage{graphicx}
\usepackage{txfonts}

\newcommand{\R}{\mathbb{R}}
\newcommand{\SSS}{\mathbb{S}}
\newcommand{\C}{\mathbb{C}}
\newcommand{\Z}{\mathbb{Z}}
\newcommand{\delcaption}[3]{\text{\footnotesize$#1$ ($#2$,\,$#3$)}}

\newcommand{\uhat}{\widehat{\mathcal{U}}^\C_+}

\newcommand{\PP}{\mathbb{P}}

\newcommand{\SL}{\operatorname{SL}}
\newcommand{\SU}{\operatorname{SU}}

\newcommand{\Ker}{\operatorname{Ker}}

\newcommand{\id}{I}

\newcommand{\bbar}{\begin{pmatrix}}
\newcommand{\ebar}{\end{pmatrix}}

\newcommand{\bdm}{\begin{displaymath}}
\newcommand{\edm}{\end{displaymath}}
\newcommand{\beq}{\begin{equation}}
\newcommand{\beqa}{\begin{eqnarray}}
\newcommand{\beqas}{\begin{eqnarray*}}
\newcommand{\eeq}{\end{equation}}
\newcommand{\eeqa}{\end{eqnarray}}
\newcommand{\eeqas}{\end{eqnarray*}}
\newcommand{\dd}{\textup{d}}

\newcommand{\calE}{\mathcal{E}}
\newcommand{\suchthat}{~|~}
\newcommand{\abs}[1]{{\lvert#1\rvert}}
\newcommand{\MINK}{{\R^{2,1}}}

\newcommand{\sym}{\mathcal{S}}

\newcommand{\Ad}{\textup{Ad}}

 \newcommand{\real}{{\mathbb R}}

\newcommand{\stilde}{\widetilde{\mathcal{S}}}
\newcommand{\hh}{\mathcal{U}}

   \newtheorem{theorem}{Theorem}[section]
   \newtheorem{proposition}[theorem]{Proposition}
   \newtheorem{corollary}[theorem]{Corollary}
   \newtheorem{lemma}[theorem]{Lemma}
   \newtheorem{definition}[theorem]{Definition}

 \theoremstyle{remark}
   \newtheorem{example}[theorem]{Example}
   \newtheorem{remark}[theorem]{Remark}

\begin{document}

\title[CMC Surfaces in Minkowski Space via Loop Groups]{Holomorphic
 Representation  of Constant Mean
Curvature Surfaces in Minkowski Space: Consequences of Non-Compactness
in Loop Group Methods}

\begin{abstract}
We give an infinite dimensional generalized Weierstrass representation 
 for spacelike
constant mean curvature (CMC) surfaces in Minkowski 3-space $\real^{2,1}$.
 The formulation
is analogous to that given by Dorfmeister, Pedit and Wu for CMC
surfaces in Euclidean space, replacing the group $SU_2$ with
 $SU_{1,1}$.  The non-compactness of the latter
group, however, means that the Iwasawa decomposition of the loop group,
used to construct the surfaces, is not global. We  prove that it
is defined on an open dense subset, after doubling the size of 
the real form $SU_{1,1}$, and prove several results concerning the behavior
of the surface as the boundary of this open set is encountered.
We then use the generalized Weierstrass representation 
to create and classify new examples of spacelike 
CMC surfaces in $\real^{2,1}$. In particular, we
classify surfaces of revolution and surfaces with screw motion 
symmetry, as well as studying another class of surfaces for which
the metric is rotationally invariant.  
\end{abstract}

\author{David Brander}
\address{Department of Mathematics, Matematiktorvet, Technical University of Denmark, DK-2800, Kgs. Lyngby, Denmark}
\email{D.Brander@mat.dtu.dk}

\author{Wayne Rossman}
\address{Department of Mathematics, Faculty of Science, Kobe University,
Japan}
\email{wayne@math.kobe-u.ac.jp}

\author{Nicholas Schmitt}
\address{GeometrieWerkstatt, Mathematisches Institut, Universit\"at 
T\"ubingen, Germany}
\email{nschmitt@mathematik.uni-tuebingen.de}



%\date{\today}

\maketitle

\section*{Introduction}
\subsection{Motivation}
It is well known that minimal surfaces in 
Euclidean $3$-space have a Weierstrass 
representation in terms of holomorphic functions, and that
the Gauss map of such a surface is holomorphic. For \emph{non}-minimal
constant mean curvature (CMC) surfaces, Kenmotsu \cite{kenmotsu}
showed that the Gauss map is harmonic, and gave a formula for
obtaining CMC surfaces from any such harmonic maps. 
On the other hand,
as a result of work by Pohlmeyer \cite{pohlmeyer}, Uhlenbeck \cite{Uhl}
and others, it became known that
harmonic maps from a Riemann surface into a symmetric space $G/H$ can be
lifted to holomorphic maps into the based loop group $\Omega G$,
satisfying a horizontality condition - see \cite{guest} for the history.  
Subsequently, Dorfmeister, Pedit
and Wu \cite{DorPW} gave a method, the so-called DPW method,
for obtaining such harmonic maps
directly from a certain holomorphic map into the complexified loop
group $\Lambda G^\C$, via the Iwasawa splitting of this group,
$\Lambda G^\C = \Omega G \cdot \Lambda ^+ G^\C$. This method has the advantage 
that the holomorphic loop group map
itself is obtained from a collection of arbitrary complex-valued holomorphic functions. Combined  with the Sym-Bobenko formula, discussed below, for obtaining a 
surface from its loop group extended frame, this gives an infinite
 dimensional 
 ``generalized Weierstrass representation'' for CMC surfaces in terms of holomorphic functions. 

Integrable systems methods have been shown to have many applications
in submanifold theory. Concerning CMC surfaces,
notable early results were the classification of CMC tori in
$\R^3$ by Pinkall and Sterling \cite{pinkallsterling}, and the 
rendering of all CMC tori in space forms in terms of theta functions by
Bobenko \cite{bobenko}.
The DPW method has led to new examples of 
non-simply-connected CMC surfaces in $\R^3$ - and other space 
forms - that have not yet been proven to exist by any other 
approach  \cite{KKRS}, \cite{KMS}, 
\cite{SKKR}.

Unsurprisingly, an analogous construction is obtained for spacelike, 
which is to say Riemannian, CMC surfaces in Minkowski space $\R^{2,1}$,
by replacing the group $SU_2$, used in the Euclidean case,
with the non-compact real form $SU_{1,1}$. However, there is a
major difference, in that the Iwasawa decomposition is not
global when the underlying group is non-compact, which has 
consequences for the global properties of the surfaces constructed.

There is already an extensive collection of work about spacelike 
non-maximal CMC 
surfaces in $\R^{2,1}$ and their harmonic (\cite{milnor})
Gauss maps.  
 Works of Treibergs 
\cite{Trei1}, Wan \cite{wan},  and Wan-Au \cite{wan2} 
 show existence of a large class of 
entire examples, which are then necessarily complete (Cheng and Yau 
\cite{cheng-yau}). Other studies, also without the loop group point
of view, include \cite{treibergschoi} and \cite{akutagawanishikawa}.
Inoguchi \cite{inoguchi} gave a loop group formulation and discussed
 finite type solutions and solutions obtained via dressing, which are
 two further methods, distinct from the DPW method employed here,
  that can also be  used for loop group type problems. 

Studying the generalized Weierstrass representation for CMC surfaces in $\real^{2,1}$
 is  interesting for various reasons: from the
viewpoint of surface theory, there is naturally a richer variety
of such surfaces, compared to the Euclidean case,
 due to the fact that not all directions are the
same in Minkowski space. 
 CMC surfaces in Minkowski space are important
in the study of classical relativity - see for example, the work of
Bartnik and Simon \cite{bartniksimon, bartnikacta}.
 The main issue addressed in those works was to give conditions which
would guarantee that surfaces obtained from a variational problem
are everywhere spacelike. The holomorphic representation studied
here is a completely different approach:
 all surfaces are, in principle,
obtained from this method and 
the surface is guaranteed
to be spacelike as long as the holomorphic loop group map takes its
values in an open dense subset of the loop group (the ``big cell'').
The surface fails to be spacelike or immersed
only when the corresponding holomorphic data encounters 
 the boundary of this dense set.
   Since
all CMC surfaces have such a representation,
understanding the behavior at this boundary potentially
gives a means to characterize the singularities.
More generally in the context of integrable systems
in geometry, this example can be thought of as a test case regarding the 
significance of the absence of a global Iwasawa decomposition,
or, more broadly, of the non-compactness of the group.

%*********************************************
\subsection{Results}
In Sections \ref{section1-wayne} and \ref{twistediwasawasection} we present the Iwasawa decomposition
 associated to the group of loops in $SU_{1,1}$.
The general case for non-compact groups had been earlier treated
by Kellersch \cite{Kell}; we provide a rather explicit
proof for our case.  
 The main new result here, which is important for
our applications, is that, after doubling  the size of the group,
by setting $G=  SU_{1,1} \, \sqcup \, i\sigma_1 \cdot SU_{1,1} $, 
where $\sigma_1$ is a Pauli matrix,
 we are able to prove that the Iwasawa splitting we need
is almost global. That is, if $\Lambda G^\C$ is the group of loops in
a complexification $G^\C$ of $G$,
 $\Lambda ^+ G^\C$ is the subgroup of loops
which extend holomorphically to the unit disc, and $\Omega G$ is the subgroup
of based loops mapping 1 to the identity, then 
\beq \label{Iwasawa1}
\Omega G \cdot \Lambda ^+ G^\C
\eeq
is an open dense subset, called the \emph{(Iwasawa) big cell}, of 
$\Lambda G^\C$. We are primarily interested in this result
in the twisted setting, described in Section \ref{twistediwasawasection}.

We also prove, in Section \ref{nickssection}, 
that, for a loop which extends meromorphically to the unit disk with exactly one
pole,  the Iwasawa decomposition can be computed explicitly using finite linear 
algebra.  This result is used for the analysis of singularities arising in CMC
surfaces.

In Section \ref{section2} we give the loop group formulation and
the DPW method for CMC surfaces in Minkowski space. 
This uses the first factor $F$ of the decomposition  $\phi = F B$,
corresponding to (\ref{Iwasawa1}),
to obtain a CMC surface from a certain
holomorphic map $\phi: \Sigma \to \Lambda G^\C$, where $\Sigma$ is a
Riemann surface. 

In Section \ref{brander-section} we examine the behavior of the
surfaces  at the boundary of the big cell.  In Theorem \ref{summarythm1}, we prove that the DPW
construction maps an open dense set $\Sigma^\circ \subset \Sigma$
to a smooth CMC surface, and that the singular set, $\Sigma \setminus \Sigma^\circ$
is locally given as the zero set of a non-constant real analytic function.

The boundary of the big cell 
is a countable disjoint union of ``small cells'', the first two of which are
of lowest codimension in the loop group, and therefore the 
most significant. We examine the behavior of the surface as points on 
the set  $\Sigma \setminus \Sigma^\circ$ which correspond to
the first two small cells are approached. In Theorem \ref{summarythm},
we prove that the surface always has finite singularities
at points which are mapped by $\phi$ to the first small cell 
(and this also occurs 
along the zero set of a non-constant real analytic function).
On the other hand, we prove that, as points mapping to the second small cell
are approached, the surface  is always unbounded and the metric blows up.

The next two sections are devoted to applications. 
There are a variety of CMC rotational surfaces 
in $\real^{2,1}$, because the rotation axes can be either timelike or spacelike or 
lightlike.  Classifications of such rotational surfaces were considered by Hano and 
Nomizu 
\cite{HanoNom} and Ishihara and 
Hara \cite{HI}, with the aim of studying 
rolling curve constructions for the profile curves, but the moduli
space was not considered.  Here we find the moduli spaces for both surfaces
of revolution and the more general class of equivariant surfaces.
In Section \ref{section4-wayne}, we 
explicitly construct and classify all spacelike CMC surfaces of revolution in 
$\R^{2,1}$.  
In particular, this results in a new family of loops for which we know the 
explicit $SU_{1,1}$-Iwasawa splitting.  We also study the 
surfaces in the associate families of the CMC surfaces of revolution, 
which we prove give all spacelike CMC surfaces with screw motion symmetry
(equivariant surfaces). In both those cases, the explicit nature of
the construction can be used to study the singularities and the end 
behaviors of the surfaces.

In Section \ref{section5-wayne} we use the Weierstrass representation  
to  construct $\R^{2,1}$ analogues of 
Smyth surfaces \cite{smyth}  (surfaces whose metrics have
a rotational symmetry), and study their properties.

%*****************************************************************

%###############################################################
%****************************************************************

\section{The Iwasawa decomposition for the untwisted loop group} \label{section1-wayne}

If $G$ is a compact semisimple Lie group, then the Iwasawa decomposition
of $\Lambda G$, proved in \cite{PreS}, is 
\beq \label{psiwasawa}
\Lambda G^\C = \Omega G \cdot \Lambda^+ G,
\eeq 
where $\Omega G$ is the set of based loops $\gamma \in \Lambda G$ such that
$\gamma(1) = 1$.  For non-compact groups, this problem was investigated by
Kellersch \cite{Kell}.  An English presentation of those results can be found 
in the appendix of \cite{BD}.  Here we restrict to $SU_{1,1}$, as it is a 
representative example, and as it has applications to CMC surface theory.

\subsection{Notation and definitions} \label{notation1}
Throughout this article we will make extensive use of the Pauli  matrices
\[
 \sigma_1 := \bbar 0 & 1 \\ 1 & 0 \ebar \; , \;\; \sigma_2 :=
 \bbar 0 & -i \\ i & 0 \ebar \; , \;\; 
 \sigma_3:= \bbar 1 & 0 \\ 0 & -1 \ebar \; .
  \] 
Let $\SSS^1$  be the unit circle 
in the complex $\lambda$-plane, 
$D_+$ the open unit disk, and 
$D_- = \{\lambda \in \C \,| \,  |\lambda | > 1 \} \cup \{ \infty \}$  the 
exterior disk in $\C\PP^1$. 
 
If $G^\C$ is any complex semisimple Lie group then $\Lambda G^\C$ denotes the
Banach Lie group of maps from $\SSS^1$ into $G^\C$ with some $H^s$-topology,
$s> 1/2$. All subgroups are given the induced topology. For any subgroup 
$\mathcal{H}$ of $\Lambda G^\C$ we denote the subgroup of constant loops,
which is to say $\mathcal{H} \cap G^\C$,
 by $\mathcal{H}^0$.

For us, $G^\C$ will be the special linear group $SL_2\C$. Now the real
form $SU_{1,1}$ is the fixed point subgroup with respect to the involution
\beq  \label{gtaudef}
\tau(x) = 
\Ad_{\sigma_3} (\overline{x^t})^{-1}.
\eeq
For our application, however, it will become clear that it is convenient to 
set 
\bdm
G := \{ x \in SL_2\C ~|~ \tau (x) = \pm x \} .
\edm
As a manifold, $G$ is a disjoint union 
$SU_{1,1} \, \sqcup \, i\sigma_1 \cdot SU_{1,1}  $, and has a complexification
$G^\C = SL_2 \C$. 
It turns out that $G$ works just as well as $SU_{1,1}$ for
our application, and this choice will mean that the Iwasawa decomposition is
almost global.  We remark that an alternative way to achieve this would have
been to set $G^\C$ to be the group $\{ x \in GL(2,\C) ~|~ \det x = \pm 1 \}$,
and in this case the appropriate real form $G$ would be just the fixed point
subgroup with respect to $\tau$.

Let $\Lambda G$ denote the subgroup of $\Lambda G^\C$ consisting of
loops with values in the  subgroup $G$.  
We extend  $\tau$ to an involution of the loop group by the formula
\beqa  \label{extendtaudef}
(\tau(x))(\lambda) &:=& \tau(x(\bar{\lambda}^{-1})) \\
&=& \sigma_3 (\overline{x(\bar \lambda^{-1})}^t)^{-1} \sigma_3. \nonumber
\eeqa
Then it is easy to verify that the definition of $\Lambda G \subset \Lambda G^\C$ 
is the analogue of $G \subset G^\C$:
\beqas
 \Lambda G &=& \{x \in \Lambda G^\C ~|~ \tau(x) = 
\pm x \} \; , \\
&=& (\Lambda G^\C)_\tau \, \sqcup \, i \sigma_1 \cdot (\Lambda G^\C)_\tau  ,
\eeqas
where $(\Lambda G^\C)_\tau = \Lambda SU_{1,1}$ is the fixed point subgroup 
with respect to $\tau$.
We want a decomposition similar to the Iwasawa decomposition (\ref{psiwasawa}),
but our group $G$ is non-compact.

\subsubsection{Normalizations for the untwisted setting}
 
Let $\triangle^+$ and $\triangle^-$ denote the sets of 
$2 \times 2$ upper triangular and lower triangular matrices, 
respectively, and $\triangle^\pm_\real$ denote the subsets with the further
restriction that the diagonal components are positive and real.
 For any lie group $X$, let $\Lambda ^\pm X$ denote the subgroup
consisting of loops which extend holomorphically to $D_\pm$.
 We start by defining some further subgroups
of the untwisted loop group $\Lambda G^\C := \Lambda SL_2\C$.
Denote the centers of the interior and exterior disks, $D_\pm$,
 by $\lambda_+ := 0$ 
and $\lambda_- := \infty$.
 Set
\[ \Lambda_{\triangle}^\pm G^\C := \{ B \in \Lambda^\pm G^ \C \, | \, 
       B(\lambda_\pm) \in \triangle^\pm \} , \]
\[ \Lambda_{\real}^+G^\C := \{ B \in \Lambda^+ G^ \C \, | \, 
       B(0) \in \triangle^+_\real \} , \]
\[
\Lambda_{I}^\pm G^ \C := \{ B \in \Lambda^\pm G^ \C \, | \, 
       B(\lambda_\pm) = I \} \; , \] 

\subsection{The Birkhoff decomposition}
To obtain the corresponding results for the twisted loop
group later, we  normalize the factors in the Birkhoff
factorization theorem of \cite{PreS}, in a certain way:
%---------------birkhoff theorem---------------
\begin{theorem}\label{thm:birkhoff} 
(Birkhoff decomposition \cite{PreS})  
Any $\phi \in \Lambda G^ \C$, has a decomposition:
\[ \phi = B_-  M
 B_+ \, , \hspace{.5cm} B_\pm \in \Lambda^{\pm}_\triangle G^ \C \,
 ,  \] 
where either
\bdm
M = \bbar  \lambda^\ell &0 \\ 
    0 & \lambda^{-\ell}  \ebar, \hspace{.5cm} \textup{or} 
    \hspace{.5cm}
M= \bbar  0 & \lambda^\ell \\ -\lambda^{-\ell} &0 \ebar,
\hspace{.5cm} \ell \in \Z \; .
\edm
The middle term, $M$, is uniquely determined by $\phi$.
The big cell $\mathcal{B}^U$, where $l=0$, is 
open and dense in $\Lambda G^ \C$, and in this case 
there is a unique factorization $\phi = \hat{B}_- M_0 \hat{B}_+$,
with $\hat{B}_\pm \in \Lambda_I^\pm G^\C$ and $M_0 \in G^\C$.
 Moreover, the map
$\mathcal{B}^U \to \Lambda^{-}_I G^ \C \times G^ \C \times
\Lambda^{+}_I G^ \C$, given by $[ \phi \mapsto (\hat{B}_-,\,  M_0 , \, \hat{B}_+)]$,
is a real analytic diffeomorphism.
\end{theorem}

\begin{proof}
The result is stated and proved in an alternative form as 
Theorem 8.1.2 and Theorem 8.7.2 of \cite{PreS}, 
without the upper and lower triangular normalization of the
constant terms, and where the middle term, $M$, is a homomorphism
from $\SSS^1$ into a maximal torus, which is to say the first type
of middle term here. That is 
\bdm
\phi =  \phi_-  \bbar \lambda^l &0\\ 0 & \lambda^{-l} \ebar  \phi_+,
\hspace{1.5cm} \phi_\pm \in \Lambda^\pm G^\C.
\edm

Such a product can be manipulated so that the constant terms of $\phi_\pm$ 
are appropriately triangular if one allows the middle term to 
become off-diagonal.
\end{proof}
%-------------------

\subsection{The untwisted Iwasawa decomposition for $G$}

Define the untwisted Iwasawa big cell 
\bdm
\mathcal{B}_{1,1}^U := \{ \phi \in \Lambda G^\C ~|~ (\tau(\phi))^{-1}\phi \in \mathcal{B}^U \}.
\edm
%----------------
\begin{theorem}\label{thm:SU11Iwa} (Untwisted 
$SU_{1,1}$ Iwasawa decomposition) 
\begin{enumerate}
\item 
The group $\Lambda G^\C$ is a disjoint union, 
\bdm
\mathcal{B}_{1,1}^U \sqcup \bigsqcup_{m \in \Z } \widehat{\mathcal{P}}_m,
\edm
where $\widehat{\mathcal{P}}_m$ are defined below at item (3).
\item 
Any element $\phi \in \mathcal{B}_{1,1}^U$ has a decomposition
\bdm
\phi = F  B, \hspace{1.5cm} 
F \in \Lambda G, \hspace{.5cm} B \in \Lambda^+_\triangle G^ \C.
\edm 
We can choose $B \in \Lambda^+_\real G^\C$,
 and then  $F$ and $B$ are uniquely determined, and 
 the product map
$\Lambda G   \times \Lambda^+_\real G^\C \to \mathcal{B}_{1,1}^U$ is a 
real analytic diffeomorphism. We call this unique decomposition \emph{normalized}.
\item 
Any element $\phi \in \widehat{\mathcal{P}}_m$ can be expressed as 
\bdm
\phi = F \hat \omega_m B, \hspace{1.5cm} 
F \in (\Lambda G^\C) _\tau, \hspace{.5cm} B \in \Lambda^+_\triangle G^ \C,
\edm 
where 
\bdm
\hat \omega_m := \begin{pmatrix}
\frac{1}{2} & \lambda^m \\ -\frac{1}{2} \lambda^{-m} & 1
\end{pmatrix}.
\edm
\item
The Iwasawa big cell $\mathcal{B}^U_{1,1}$ is an open dense set of $\Lambda G^\C$.
The complement of the big cell is locally given as the zero set of a non-constant
real analytic function $g: \Lambda G^\C \to \C$.
\end{enumerate}
\end{theorem}
%------------

The proof of Theorem \ref{thm:SU11Iwa} is a consequence of the following lemma:
\begin{lemma}\label{lem:preSU11Iwa}
If $\psi \in \Lambda G^ \C$ satisfies $(\tau(\psi))^{-1} 
= \psi$, then \[ \psi = (\tau(B_+))^{-1} (\pm I) B_+ \;\; 
\text{or} \;\;\; \psi = (\tau(B_+))^{-1} 
\bbar 0 & \lambda^m \\ -\lambda^{-m} & 0 \ebar  B_+ \] for some 
uniquely determined integer $m$, and for some 
$B_+ \in \Lambda_\triangle ^+ G^\C$.  
\end{lemma}

\begin{proof}
Consider the two cases for the Birkhoff splitting of $\psi$ given in Theorem \ref{thm:birkhoff}. First, if
$\psi = B_- \text{diag}(\lambda^k,\lambda^{-k}) B_+$,  
then \begin{equation}\label{B-definition}
      B = B_+ \tau(B_-) =\begin{pmatrix}
      a & b \\ c & d \end{pmatrix}
     \end{equation}
is an 
element of $\Lambda_\triangle^+ G^ \C$, and the assumption that 
$(\tau(\psi))^{-1} = \psi$ is equivalent to the equation
\[
 \begin{pmatrix}
 a^* \lambda^{-k} & 
-c^* \lambda^k \\ 
-b^* \lambda^{-k} & 
d^* \lambda^k
 \end{pmatrix}
=
 \begin{pmatrix}
 a \lambda^k & b \lambda^k \\ 
c \lambda^{-k} & d \lambda^{-k}
 \end{pmatrix} \; . 
\]  It follows that $b$ and $c$ are both identically zero, 
that $a, d$ are constant and real, and that $k=0$.  So 
$B = \text{diag}(\alpha,\alpha^{-1}) (\pm I) 
  \text{diag}(\alpha,\alpha^{-1})$ for some constant 
$\alpha > 0$.  Then 
$\psi = (\tau (\psi))^{-1}= (\tau (B_+))^{-1}(\tau(B_-))^{-1} = \tau(\tilde B_+)^{-1} 
   (\pm I) \tilde B_+$, where 
$\tilde B_+ = \text{diag}(\alpha^{-1},\alpha) B_+$.  

Now consider the case  $\psi = B_- 
{\tiny{\bbar 0 & \lambda^k\\ -\lambda^{-k} & 0 \ebar}} B_+$.  Proceeding as before,
we have
\[
 \begin{pmatrix}
 a^* & 
-c^* \\ 
-b^* & 
d^*
 \end{pmatrix}
 \begin{pmatrix}
 0 & \lambda^k \\ -\lambda^{-k} & 0
 \end{pmatrix}
=
 \begin{pmatrix}
 0 & \lambda^k \\ -\lambda^{-k} & 0
 \end{pmatrix}
 \begin{pmatrix}
 a& b \\ c & d
 \end{pmatrix} \; ,
\] where $B$ is as in \eqref{B-definition}.  
It follows that $\bar a = d$ is constant and 
$|a|=1$, and $b \cdot c$ is identically zero.  Further, 
when $k<0$, then $b=0$ and  $c = c^* \lambda^{-2k}$ 
with a finite expansion in $\lambda$ of the form 
$c = c_1 \lambda^1 + ... + c_{-2k-1} \lambda^{-2k-1}$,
while, on the other hand, if $k \geq 0$, we have that
$c=0$ and 
$b=b^* \lambda^{2k}$, with
 $b = b_0 \lambda^0 + ... + b_{2k} \lambda^{2k}$.  

Setting $\tilde{B}_+ = y B_+$ and $\tilde{B}_- = B_- x^{-1}$ 
then the requirements that  $\tilde{B}_+ \in \Lambda_\triangle^+ G^\C$
and that $\psi = (\tau(\tilde{B}_+))^{-1} 
\bbar 0 & \lambda^m \\ -\lambda^{-m} & 0 \ebar  \tilde{B}_+$
 will be satisfied
if we can choose $y \in \Lambda^+_\triangle G^\C$ and 
$x \in \Lambda^-_\triangle G^\C$ with the properties:
\[
 x^{-1} \bbar 0 & \lambda^k \\ -\lambda^{-k} & 0 \ebar y = 
\bbar 0 & \lambda^k \\ -\lambda^{-k} & 0 \ebar \; , \hspace{1cm}
B = y^{-1} \tau(x). 
\]
 Set 
\[
 y =  \begin{pmatrix}
 \sqrt{a}^{-1} & y_1 \\ y_2 & \sqrt{a} 
 \end{pmatrix}
\; , \hspace{1cm}
 x^{-1} =  \begin{pmatrix}
 \sqrt{a}^{-1} & x_1 \\ x_2 & \sqrt{a}
 \end{pmatrix} \; , 
\]  
then when $k \geq 0$, we can take
$(y_1,y_2,x_1,x_2)=(-\sqrt{a} b/2,0,0,-\sqrt{a} b \lambda^{-2k}/2)$.
When $k < 0$, we take
$(y_1,y_2,x_1,x_2)=(0,-c/(2 \sqrt{a}),-c \lambda^{2k}/(2 \sqrt{a}),0)$.  
\end{proof}
%--------------------------

\noindent \textbf{Proof of Theorem \ref{thm:SU11Iwa}}
\begin{proof}
Take any $\phi \in \Lambda G^ \C$.  Set 
$\psi := \tau(\phi)^{-1} \phi$. Then $(\tau(\psi))^{-1} 
= \psi$ and so we can apply Lemma \ref{lem:preSU11Iwa}, which 
implies that 
 \[ \psi = (\tau(B_+))^{-1} 
\tau(\hat \omega)^{-1} \hat \omega B_+ \; , \] 
where $\hat \omega$ is (uniquely) one of the following:
\bdm \hat \omega_+ = I, \hspace{1cm}
\hat \omega_- = \bbar 0 & 1 \\ -1 & 0 \ebar, \hspace{1cm} 
\hat \omega_m = \begin{pmatrix}
\tfrac{1}{2} & \lambda^m \\ -\tfrac{1}{2} \lambda^{-m} & 1
\end{pmatrix},
\edm
$m \in \Z$, and $B_+ \in \Lambda^+_\triangle G^\C$.
 To see this, 
compute that $(\tau(\hat \omega_+))^{-1} \hat \omega_+ = I$, 
$\tau(\hat \omega_-)^{-1} \hat \omega_-  = -I$ and 
$\tau(\hat \omega_m)^{-1} \hat \omega_m = 
{\tiny{\bbar 0 & \lambda^m \\ -\lambda^{-m} & 0 \ebar}}$.  

Hence 
\bdm
\phi = \hat{F} \hat \omega B_+,
\edm
 where $\hat F = \tau(\phi) 
\tau(\hat \omega B_+)^{-1}$.  Now $\psi = \tau(\hat \omega B_+)^{-1} \cdot 
\hat \omega B_+$ is equivalent to the equation $\tau(\hat F)= \hat F$, and
so $\hat F \in (\Lambda G^\C)_\tau$. 

To prove item (2) of the theorem, note that $\phi \in \mathcal{B}^U_{1,1}$
if and only if $(\tau(\phi))^{-1} \phi \in \mathcal{B}^U$, and this 
corresponds to $\hat \omega = \hat \omega_\pm$, by the construction in
Lemma \ref{lem:preSU11Iwa}.  Since $\tau (\hat \omega _\pm) = \pm \hat \omega _\pm$, 
$\phi = F  B_+$, with  $F := \hat F \hat \omega_\pm$, is the required decomposition.
The uniqueness and the diffeomorphism property follow from the corresponding 
properties on the big cell in Theorem \ref{thm:birkhoff}.

Item (3) has already been proved, and the disjointness property of item (1)
follows from the uniqueness of the middle term in the Birkhoff Theorem.

To prove item (4) note that,
by definition, $\mathcal{B}^U_{1,1} = h^{-1} (\mathcal{B}^U)$, where 
$h: \Lambda G^\C \to \Lambda G^\C$ takes $\phi \mapsto (\tau(\phi))^{-1} \phi$.
It is shown in \cite{DorPW} that the Birkhoff big cell $\mathcal{B}^U$ is given
 as the complement of the zero set of a non-trivial holomorphic section
$\mu$ (called $\tau$ in \cite{DorPW}) 
 of  the holomorphic line bundle $\psi^* Det^* \to \Lambda G^\C$,
where $\psi$ is a composition of holomorphic maps 
$\Lambda G^\C \to GL_{res}(H) \to Gr(H)$,
and $Det^* \to Gr(H)$ is the dual of the determinant line bundle.
Hence the Iwasawa big cell $\mathcal{B}^U_{1,1}$ is given as the complement of
the zero set of the section $h^* \mu$, locally represented by a real analytic
function $g: \Lambda G^\C \to \C$.  The complement of such a 
zero set is either open and dense or empty, and the big cell is not empty,
as it contains the identity.
\end{proof}

\begin{remark} A similar procedure can be used to prove the $SU_2$ Iwasawa splitting.
In that case, as a consequence of the compactness of the group,
everything is much simpler and the small cells $\hat{\mathcal{P}}_m$
do not appear.
\end{remark}

%**************************************************************

\subsection{Explicit Iwasawa factorization of Laurent loops} \label{nickssection}

Computing the Iwasawa
factorization explicitly is not possible in general.  However, if
$X \in \mathcal{B}^U_{1,1}$ extends meromorphically to
 the unit disk, with  just one pole at $\lambda = 0$,
then the Iwasawa decomposition can be computed by finite linear algebra.
To show this we will define a linear operator on a finite dimensional vector space
whose kernel corresponds to the $G$ factor of $X$.  

For $-\infty < p \le q \le \infty$, denote the
vector space of formal Laurent series by 
\[
\Lambda_{p,q} = \left\{
  \left.
    \textstyle\sum_{j=p}^q a_j\lambda^j
  \right|
    a_j\in M_{2 \times 2}\C  \right\},
\]
and let $P_{p,\,q}:\Lambda_{-\infty,\infty}\to\Lambda_{p,q}$
be the projection
\begin{equation*}
  P_{p,q}\left(\textstyle\sum_{j=-\infty}^\infty a_j\lambda^j\right)
    = \textstyle\sum_{j=p}^{q}a_j\lambda_j.
\end{equation*}
Define the anti-involution $\rho$ on $\Lambda_{-\infty,\infty}$
by, for $W \in \Lambda_{-\infty,\infty}$,
\[
(\rho{W})(\lambda) =
 \sigma_3 \overline{W\left(1/\overline{\lambda}\right)}^t \sigma_3 .
\]
Note that if $W(\lambda)$ is
an invertible matrix, then $\rho$ is the composition of 
$\tau$ with the matrix inverse operation.

For any $X \in \Lambda G^\C$, define a linear 
map
$\mathcal{L}_X: \Lambda_{-\infty,\infty}\to
   \Lambda_{-\infty,-1}\oplus \overline{\Lambda_{-\infty,-1}}\oplus \C^4$ 
by
\begin{equation}
\begin{split}
& \mathcal{L}_X(W) = \\
 &\quad\ \ %
 \left(
     P_{-\infty,-1}(W X),\ \overline{P_{-\infty,-1}(\textup{adj}(\rho W) X)},\,
    \left. \left( P_{0,0}(WX) - \overline{P_{0,0}(\textup{adj}(\rho W) X)}\right)\right|_{11},\,
 \right.\\
  &\quad\ \ %
  \left.
      \left. \left( P_{0,0}(WX) - \overline{P_{0,0}(\textup{adj}(\rho W) X)}\right) \right|_{22},\,
      P_{0,0}(WX)\Big|_{21},
      \left. \overline{P_{0,0}(\textup{adj}(\rho W) X))}\right|_{21}
  \right),\quad
\end{split}
\end{equation}
where 
$\text{adj}$ gives the adjugate matrix and
the subscripts $ij$ refer to matrix entries.
The map $\mathcal{L}_X$ is clearly complex linear.  

\begin{lemma}
\label{lem:laurent-iwasawa-kerL}
Let $n\in\Z_{\ge 0}$ and $X\in {\Lambda G^\C} \cap\Lambda_{-n,\infty}$.
Suppose $X$ lies in the big cell, and let 
$X = FB$ be its normalized $\SU_{1,1}$-Iwasawa factorization.  
Then
\begin{enumerate}
\item
\label{lem:laurent-iwasawa-kerL-1}
$\Ker\mathcal{L}_I = \C \cdot \id$ and 
$\Ker\mathcal{L}_{i\sigma_2} = \C\cdot\sigma_1$.  
\item
\label{lem:laurent-iwasawa-kerL-2}
If $F\in (\Lambda G^\C)_\tau$, then $\Ker\mathcal{L}_X = \C \cdot F^{-1}$.
\item
\label{lem:laurent-iwasawa-kerL-3}
If $F\in i \sigma_1 \cdot (\Lambda G^\C)_\tau$, then 
$\Ker\mathcal{L}_X = \C \cdot (F\sigma_3)^{-1}$.
\end{enumerate}
\end{lemma}

\begin{proof}
Let $X\in {\Lambda G^\C} \cap\Lambda_{-n,\infty}$.
By the definition of $\mathcal{L}_X$, $W\in\Ker\mathcal{L}_X$ if and only if
for some $p,\,q\in\C$,
\bdm
P_{-\infty,0}(W X) =
  \begin{pmatrix}p & c_1 \\ 0 & q\end{pmatrix}
\;\;\; \text{and} \;\;\; 
P_{-\infty,0}(\textup{adj}(\rho W) X) =
  \begin{pmatrix}\overline{p} & c_2 \\ 0 & \overline{q}\end{pmatrix},
\edm
where $c_i \in \C$.
It follows that
\bdm
\begin{split}
&\Ker\mathcal{L}_I = \C \cdot \id\text{ and }
\Ker\mathcal{L}_{i\sigma_2} = \C\cdot\sigma_1 , \\
&\Ker\mathcal{L}_{XB} = \Ker\mathcal{L}_{X}
   \text{ for all $B\in \Lambda_{\R}^{+}G^{\C}$,}\\
&\Ker\mathcal{L}_{FX} = \Ker\mathcal{L}_{X}\cdot F^{-1}
   \text{ for all $F\in (\Lambda G^\C)_\tau$.}
\end{split}
\edm

Statement \ref{lem:laurent-iwasawa-kerL-2} follows from 
\[
\Ker\mathcal{L}_{FB} = \Ker\mathcal{L}_{\id} \cdot F^{-1} = \C\cdot F^{-1}.
\]
If $F\in i \sigma_1 \cdot (\Lambda G^\C)_\tau$, then 
$F = G i \sigma_2$ for some $G\in (\Lambda G^\C)_\tau$, and 
statement \ref{lem:laurent-iwasawa-kerL-3} follows from 
\[
\Ker\mathcal{L}_{G i \sigma_2 B} = \Ker\mathcal{L}_{i \sigma_2} 
   \cdot G^{-1} =
   \C\cdot i \sigma_1 G^{-1}
   =\C\cdot \sigma_3 F^{-1}
   =\C\cdot (F\sigma_3)^{-1}.
\]
\end{proof}
Let $X$ be a Laurent loop in the big cell, of pole order 
$n \in \Z_{\geq 0}$ at $\lambda=0$ and with no other singularities on the unit disk.
Let  $X=FB$ be the $\SU_{1,1}$-Iwasawa 
decomposition.  Then $F$ is a Laurent loop in 
$\Lambda_{-n,n}$, because $XB^{-1}=F$ has a pole of 
order $n$ at $\lambda=0$ and $\tau(F)= \pm F$.  
In fact, we have the following theorem: 

\begin{theorem}\label{thm:laurent-iwasawa-kerL}
Lemma \ref{lem:laurent-iwasawa-kerL}
provides an explicit construction of the
normalized $\SU_{1,1}$-Iwasawa 
decomposition of any 
$X\in {\mathcal{B}_{1,1}^U} \cap\Lambda_{-n,\infty}$
by finite linear methods.  In particular, let $X = F B$ be 
the $\SU_{1,1}$-Iwasawa decomposition.  Then
\begin{enumerate}
\item $F$ is a Laurent loop in $\Lambda_{-n,n}$ if and only if $X$ 
    extends meromorphically to the 
    unit disk, with pole of order $n$ at $\lambda = 0$, 
    and no other poles.  
\item In this case, the two conditions that $F \in \Lambda G$ and 
   that $B \in \Lambda^+_\real G^\C$ form an algebraic 
    system on the coefficients of $F^{-1}$ with a unique solution.
\end{enumerate}
\end{theorem}

\begin{proof}
Compute $W\in\Ker\mathcal{L}_X\setminus\{0\}$.  This involves 
solving a complex linear system with $16n+4$ equations and 
$8n+4$ variables.  

That $\det W$ is $\lambda$-independent can be seen as follows: Since 
$W$ solves the linear system, $WX$ and $\textup{adj}(\rho W) X$ 
are in $\Lambda _{0,\infty}$, and so $\det W(\lambda)$ and 
$\det \overline{W(1/\bar \lambda)}$ are holomorphic in the 
unit disk.  In particular, $\det W(\lambda)$ is holomorphic on 
$\C \cup \{ \infty \}$, and so is constant.

Thus, multiplying by a constant scalar if necessary, we may, and do, 
assume $\det W \equiv 1$.  

By Lemma \ref{lem:laurent-iwasawa-kerL},
$((i \sigma_3)^kW)^{-1}$ is the $\Lambda G$ factor
of the
normalized $\SU_{1,1}$-Iwasawa decomposition of $X$
for some $k\in\{0,\dots,3\}$.
\end{proof}

For the simplest case, when  $X$ is a constant loop, the 
linear system in the proof of Theorem \ref{thm:laurent-iwasawa-kerL} 
gives the following corollary:

\begin{corollary}  \label{factcor}
For $X \in \SL_2 \C$, the $\SU_{1,1}$-Iwasawa decomposition 
has three cases:
\begin{enumerate}
\item When $|X_{11}|>|X_{21}|$, there exist $u,v,\beta \in \C$ 
and $r \in \R^+$ such that $u \bar u - v \bar v = 1$ and 
\[ X = 
 \begin{pmatrix} u & v \\ \bar v & \bar u \end{pmatrix}
 \begin{pmatrix} r & \beta \\ 0 & r^{-1} \end{pmatrix} \; . 
\]
\item When $|X_{11}|<|X_{21}|$, there exist $u,v,\beta \in \C$ 
and $r \in \R^+$ such that $u \bar u - v \bar v = -1$ and 
\[ X = 
 \begin{pmatrix} u & v \\ -\bar v & -\bar u \end{pmatrix}
 \begin{pmatrix} r & \beta \\ 0 & r^{-1} \end{pmatrix} \; . 
\]
\item When $|X_{11}|=|X_{21}|$, there exist $\theta,\gamma \in \C$, 
$r \in \R^+$ and $\beta \in \C$ such that 
\[ X = 
 \begin{pmatrix} e^{i \theta} & 0 \\ e^{i \gamma} & e^{-i \theta} \end{pmatrix}
 \begin{pmatrix} r & \beta \\ 0 & r^{-1} \end{pmatrix} \; . 
\]
\end{enumerate}
\end{corollary}

%------------------------------------

\section{Iwasawa factorization in the twisted loop group} \label{twistediwasawasection}

\subsection{Notation and definitions for the twisted loop group} \label{notation}

As before, we set $G = SU_{1,1} \, \sqcup \, i\sigma_1 \cdot SU_{1,1} $, but
from now on we work in  the twisted loop group
\[
\hh^\C := \Lambda G_\sigma^\C := \{ x \in \Lambda G^\C | \, 
       \sigma (x) = x \} \; , 
\]
where the involution $\sigma$ is defined, for a loop $x$, by
\bdm
(\sigma(x))(\lambda) := \Ad_{\sigma_3} \, x(-\lambda).
\edm
We will also refer to three further subgroups of $\hh^\C$, 
\beqas
\begin{aligned}
& \hh^\C_\pm :=   \{ B \in \hh^\C \, | \, 
       B \text{ extends holomorphically to } D_\pm \}, \\
& \uhat := \{ B \in \hh^\C_+ ~|~ B(0) = \tiny{\bbar \rho & 0 \\ 0 & \rho^{-1} \ebar}, ~\rho \in \real, ~\rho>0 \}. 
\end{aligned}
\eeqas 
     
We extend  $\tau$ to an involution of the loop group by the formula
\bdm
(\tau(x))(\lambda) := \tau(x(\bar{\lambda}^{-1})).
\edm
The ``real form'' is 
\beqas
 \hh &:=& \Lambda G_\sigma = \{F \in \Lambda G^\C_\sigma ~|~ \tau(F) = 
\pm F \} \; , \\
&=& \hh_\tau \, \sqcup \, \Psi(i \sigma_1) \cdot \hh_\tau,
\eeqas
where $\hh_\tau$ is the fixed point subgroup of $\tau$, and 
$\Psi(i \sigma_1) = \tiny{\bbar 0 & i \lambda \\ i \lambda^{-1} & 0 \ebar}$ 
(see (\ref{psiisom}) below).

For any Lie group $A$, let $Lie(A)$ denote its Lie algebra. 
We use the same notations $\sigma$ and $\tau$ for the infinitesimal
versions of the involutions, which are given on $Lie(\Lambda G^\C)$ by
\bdm
(\sigma (X))(\lambda) := \Ad_{\sigma_3} \, X(-\lambda), \hspace{1cm}
(\tau (X))(\lambda) := -\Ad_{\sigma_3} \overline{X^t(\bar{\lambda}^{-1})}.
\edm
We have
 $Lie(\hh^\C) = \{X = \sum X_i \lambda^i ~|~ X_i \in \mathfrak{sl}_2\C,~ \sigma(X) = X \}$, and $Lie(\hh)$ is the subalgebra consisting of elements fixed by $\tau$. 
 The convergence condition of these series depends on the topology used.

For practical purposes, we should note that $\hh^\C$ and $Lie(\hh^\C)$ consist
of loops {\tiny{$\bbar a&b\\c&d \ebar$}} which take values in $SL_2\C$ and 
$\mathfrak{sl}_2 \C$ respectively, and such that the coefficients $a$ and $d$
are even functions of the loop parameter $\lambda$, whilst $b$ and $c$ are odd functions of $\lambda$.  $\hh_{\pm}^\C$ and $Lie(\hh_{\pm}^\C)$ are the elements which 
have the further condition that only non-negative or non-positive exponents of $\lambda$
appear in their Fourier expansions.
  For a scalar-valued function $x(\lambda)$, we use the notation 
\bdm
x^*(\lambda) := \overline{x(\bar{\lambda}^{-1}}).
\edm
 Then for the real form $\hh$ we have
\beq \label{matrixforms}
\hh_\tau = \Big\{ \bbar a&b\\ b^* & a^* \ebar \in \hh^\C \Big\}, \hspace{0.8cm}
\Psi(i \sigma_1) \cdot \hh_{\tau} = \Big\{ \bbar a&b\\ -b^* & -a^* \ebar \in \hh^\C \Big\},
\eeq
and the analogue for the Lie algebras.

\subsection{The Iwasawa decomposition for $SU_{1,1}$} \label{mainiwasawa}
To convert  Theorem \ref{thm:SU11Iwa} to the twisted setting, we 
use the isomorphism 
 from the untwisted to the twisted loop group, defined by
\beq \label{psiisom}
\Psi: \Lambda G^\C \to \Lambda G^\C_\sigma, \hspace{1cm}
\begin{pmatrix}
a(\lambda) & b(\lambda) \\ c(\lambda) & d(\lambda) 
\end{pmatrix} \mapsto 
\begin{pmatrix}
a(\lambda^2) & \lambda b(\lambda^2) \\ 
\lambda^{-1} c(\lambda^2) & d(\lambda^2) 
\end{pmatrix}.
\eeq
We define the Birkhoff big cell in $\Lambda G^\C_\sigma$ by 
$\mathcal{B}:= \Psi (\mathcal{B}^U)$. The Birkhoff factorization
theorem, Theorem \ref{thm:birkhoff}, then translates to the assertion
 that  $\mathcal{B} = \hh^\C_- \cdot  \hh^\C_+$, and that
 this  is an open dense subset of $\hh^\C$.

Define the \emph{$G$-Iwasawa big cell} for $\hh^\C$ to be
the set
\bdm
\mathcal{B}_{1,1} := \{ \phi \in 
\hh^\C ~|~ \tau(\phi)^{-1} \phi \in \mathcal{B}\}. 
\edm
It is easy to verify
 that $\tau = \Psi^{-1} \circ \tau \circ \Psi$, and this implies that
$\Psi$ maps $\mathcal{B}_{1,1}^U$ to $\mathcal{B}_{1,1}$.

To define the small cells, we first set, for a positive integer $m \in \Z^+$,
\[ \omega_{m} = \begin{pmatrix}
1 & 0 \\ \lambda^{-m} & 1
\end{pmatrix} \; , \;\; \text{$m$ odd} \;\; ; \;\;\; 
\omega_{m} = \begin{pmatrix}
1 & \lambda^{1-m} \\ 0 & 1 
\end{pmatrix} \; , \;\; \text{$m$ even.} \]  
The \emph{$n$-th small cell} is defined to be
\beq \label{pndecomp}
\mathcal{P}_n := \hh_\tau \cdot \omega_n \cdot \hh^\C_+.  
\eeq
Note that elements of  $\Omega G$, in the Iwasawa decomposition (\ref{psiwasawa}), correspond naturally to elements of
 the left coset space $\Lambda G /G$.  
For the twisted loop group, $\hh$, the role of $\Omega G$ is 
effectively played by $\hh/\hh^0$.

\pagebreak
%------------------------
\begin{theorem}\label{section1-thm:SU11Iwa} ($SU_{1,1}$ Iwasawa 
decomposition) 
\begin{enumerate}
\item\label{mainthm-(1)}
The group $\hh^{\C}$ is a disjoint union
\beq \label{globaldecomp}
 \hh^{\C} =  
\mathcal{B}_{1,1}  \sqcup \bigsqcup_{m \in \Z^+ } \mathcal{P}_m.
 \eeq
\item \label{mainthm-(2)}
Any loop $\phi \in \mathcal{B}_{1,1}$ can be expressed as 
\beq \label{bciwasawa}
\phi = F B,
\eeq
for $F \in \hh$ and $B \in \hh^\C_+$. The factor $F$ is unique up to right
multiplication by an element of $\hh^0$. The factors are unique
if we require that $B \in \uhat$,  and then the product map
$\hh \times \uhat \to \mathcal{B}_{1,1}$ is a 
real analytic diffeomorphism.

\noindent
 \item\label{mainthm-(3)}
The Iwasawa big cell, $\mathcal{B}_{1,1}$, is an open dense subset of
 $\hh^{\C}$. The complement of $\mathcal{B}_{1,1}$ in $\hh^{\C}$ is locally
 given as the zero set of a non-constant real analytic function $g: \hh^{\C} \to \C$.
\end{enumerate}
\end{theorem}
%----------------------
\begin{proof}
The theorem follows from the untwisted statement,
Theorem \ref{thm:SU11Iwa}.  Under the isomorphism $\Psi$, given by (\ref{psiisom}),
 $\hat \omega_+$ stays the same, $\hat \omega_-$ becomes 
${\tiny{\bbar 0 & \lambda \\ -\lambda^{-1} & 0\ebar}}$, 
and the $\hat \omega_m$ appear  only for odd $m$.  
  Then, noting that, for $m>0$, 
\[ (i \sigma_3) \hat \omega_m (-i \sigma_3) = \begin{pmatrix}
1 & 0 \\ \lambda^{-m} & 1
\end{pmatrix} B_+ \; , \hspace{.7cm} B_+ = \begin{pmatrix}
1/2 & -\lambda^m \\ 0 & 2
\end{pmatrix} \in \Lambda^+ G^ \C, \] 
 and, for $m<0$,
\[ \hat \omega_m = \begin{pmatrix}
1 & \lambda^m \\ 0 & 1
\end{pmatrix} B_+ \; , \hspace{.7cm} B_+ = \begin{pmatrix}
1 & 0 \\ -\lambda^{-m}/2 & 1 
\end{pmatrix} \in \Lambda^+ G^ \C, \] 
and that $B_+$ can be 
absorbed into the right-hand $\hh^\C_+$ factor 
 of any splitting, we can 
replace, in Theorem \ref{thm:SU11Iwa}, the above 
$\hat \omega_\pm$ and $\hat \omega_m$ respectively with the matrices $I$, 
$\tiny{\begin{pmatrix} 0 & \lambda\\-\lambda^{-1} & 0 \end{pmatrix}}$
 and 
the $\omega_m$ defined in Section \ref{section2}. 
This gives the small cell factorizations of  (\ref{pndecomp}) of 
Theorem \ref{section1-thm:SU11Iwa}.  
The big cell factorization of item \ref{mainthm-(2)} follows from
the observation that 
\bdm
\tau \bbar 0 & \lambda \\ -\lambda^{-1} & 0\ebar = 
- \bbar 0 & \lambda \\ -\lambda^{-1} & 0\ebar,
\edm
so that elements with this middle term can be represented as
$\phi = F B$, with $\tau(F) = -F$, that is, $F \in \Psi(i \sigma_1)  \hh_\tau \subset \hh$. 

The diffeomorphism property on the big cell,
 the disjoint union property, item \ref{mainthm-(1)}, and 
 item \ref{mainthm-(3)}
 follow from the corresponding statements in
Theorem \ref{thm:SU11Iwa}.
\end{proof}

\begin{corollary} \label{cor1}
The map $\pi: \mathcal{B}_{1,1} \to \hh/\hh^0$ given by
$\phi \mapsto [F]$, derived from
 (\ref{bciwasawa}), is a real analytic
projection.
\end{corollary}

%----------------------
\begin{remark} \label{psidefremark}
 The density of the big cell can also be seen explicitly as follows:
 consider the continuous family of loops 
\[
\psi^m_z := \begin{pmatrix} 1 & 0 \\ 
z \lambda^{-m} & 1 \end{pmatrix}, ~~\textup{$m$ odd}; \hspace{.7cm}
 \psi^m_z := \begin{pmatrix} 1 &  z \lambda^{-m+1} \\ 
0 & 1 \end{pmatrix}, ~~\textup{$m$ even}.
\]
Now $\psi^m_1 = \omega_m$, but for
$|z| \neq 1$, $\psi_z$ is in the big cell and has the Iwasawa decomposition:
$\psi^m_z = F^m_z \cdot B^m_z$, where, for odd values of $m$, 
\bdm
F^m_z =   \frac{1}{\sqrt{1-z\bar{z}}} 
\begin{pmatrix} 1 & \bar{z} \lambda ^m\\ z \lambda^{-m} & 1 
\end{pmatrix}, \hspace{.7cm}
B^m_z =  \frac{1}{\sqrt{1-z\bar{z}}}
\begin{pmatrix} 1-z \bar{z} & - \bar{z}\lambda ^m\\ 0 & 1 
\end{pmatrix},
\edm 
and, for even values of $m$:
\bdm
\psi^m_z = \textup{Ad}_{\sigma_1} \psi_z^{m-1} =
   \textup{Ad}_{\sigma_1} F^{m-1}_z \cdot \textup{Ad}_{\sigma_1} B^{m-1}_z.
\edm
If ${\phi}_0$ is any element of $\mathcal{P}_m$,  then 
it has a decomposition $\phi_0 = F_0 \omega_m B_0$, in 
accordance with (\ref{pndecomp}).
 Now define
the continuous path, for $t \in \real$, 
$\hat{\phi}_t = F_0 \psi^m_t B_0$.  
Then $\hat{\phi}_1 = \phi_0$, but for $t \neq 1$, 
 $\hat{\phi}_t = F_0 F^m_t B^m_t B_0$, which is in the big cell.
So $\hat{\phi}_t$ gives a family of elements in the big cell which
are arbitrarily close to $\phi_0$ as $t \to 1$.  
\end{remark}

\subsection{A factorization lemma}

Later, in Section \ref{brander-section},
we will use the following explicit factorization for an element of the
form $B\omega_m^{-1}$, for  $B \in \hh^\C_+$, 
and $m = 1$ or $2$.

\begin{lemma} \label{switchlemma}
Let $B = \begin{pmatrix} a & b \\ c & d \end{pmatrix} = 
\begin{pmatrix}
\sum_{i = 0}^\infty a_i \lambda^i &
  \sum_{i = 1}^\infty b_i \lambda^i \\
  \sum_{i = 1}^\infty c_i \lambda^i &
  \sum_{i = 0}^\infty d_i \lambda^i \end{pmatrix}$ 
be any element of $\hh^\C_+$. Then there 
exists a factorization
\beq \label{switchfact}
B \omega_1^{-1} = X \widehat{B},
\eeq
where $\widehat{B} \in \hh^\C_+$ and $X$ is of one of 
the following three
forms: 
\bdm
k_1 = \begin{pmatrix} u & v \lambda \\ \bar{v} \lambda^{-1} & \bar{u} 
\end{pmatrix}, \hspace{.5cm}
k_2 =\begin{pmatrix} u & v \lambda \\ -\bar{v} \lambda^{-1} & -\bar{u} 
\end{pmatrix}, \hspace{.5cm}
\omega_1^\theta = \begin{pmatrix} 1 & 0 \\ e^{i\theta} \lambda^{-1} & 1 
\end{pmatrix}, 
\edm
where $u$ and $v$ are constant in $\lambda$ and can be chosen so 
that the matrix
has determinant one, and $\theta \in \real$. 
The matrices $k_1$ and $k_2$ are in $\hh$, 
and their components  satisfy the equation 
\beq \label{uveqn}
\frac{|u|}{|v|} = |b_1-a_0||a_0| \; . \\
\eeq
The third form occurs if and only if
$B \omega_1^{-1}$ is in the first small cell, $\mathcal{P}_1$,
and the three cases  correspond 
to the cases $|(b_1-a_0)a_0|$
greater than, less than or equal to 1, respectively.

The analogue holds replacing $\omega_1$ with $\omega_2$, the matrices
$k_i$ and $\omega_1^\theta$ with $\Ad_{\sigma_1} k_i$ and 
$\Ad_{\sigma_1} \omega_1^\theta$, and replacing $\mathcal{P}_1$
with $\mathcal{P}_2$, and Equation \eqref{uveqn} with 
\beq \label{uveqn2}
\frac{|u|}{|v|} = |c_1-d_0||d_0| \; . 
\eeq
\end{lemma}
%--------------
\begin{proof}
The second statement, concerning $\omega_2$, is obtained trivially from the 
first, because  $\omega_2 = \Ad_{\sigma_1} \omega_1$, so we can get the
factorization by applying the 
homomorphism $\Ad_{\sigma_1}$ to both sides of \eqref{switchfact}.  

To obtain the factorization \eqref{switchfact}, note that under the isomorphism
given by (\ref{psiisom}), $\omega_1^{-1}$ becomes $\bbar 1 & 0 \\ -1 & 1 \ebar$,
so the untwisted form of $B \omega_1^{-1}$ has no pole on the unit disc, and the
factorization can be obtained by factoring the constant term, using Corollary
\ref{factcor}. 

Alternatively, one can write down  explicit expressions as follows:
 for the cases 
 $|(b_1-a_0)a_0|^{\varepsilon} >1$, where $\varepsilon = \pm 1$,
  the factorization is given by
\beq \label{explicitfact}
\begin{split}
&X = 
\begin{pmatrix} u & v \lambda \\ \varepsilon \bar{v} \lambda^{-1} & \varepsilon \bar{u} 
\end{pmatrix} ,  \\
&\widehat{B} = 
\begin{pmatrix} -\varepsilon \bar{u}b \lambda^{-1} + dv + \varepsilon \bar{u}a -vc \lambda ~&~
   b \varepsilon \bar{u} - v d \lambda \\
   \varepsilon \bar{v} b \lambda^{-2} -(\varepsilon \bar{v} a + u d)\lambda^{-1} + u c ~&~
     -b \varepsilon \bar{v}\lambda^{-1} + u d \end{pmatrix}. 
\end{split}
\eeq
One can choose $u$ and $v$ so that $\varepsilon (u \bar{u} - v \bar{v}) = 1$ and 
such that $\widehat{B} \in \hh^\C_+$, the latter condition being 
assured by the requirement that 
$\frac{u}{\bar{v}} = \varepsilon (b_1 -a_0)a_0$.
It is straightforward to verify that $X \widehat{B} =  B \omega_1^{-1}$.

For the case $|(b_1-a_0)a_0| =1$, substitute $\bar u$ for $\varepsilon \bar u$ and
 $-\bar{v}$ for
$\varepsilon \bar{v}$  in the above expression, and choose
$\frac{u}{\bar{v}} =  (a_0-b_1)a_0$.
One can choose $u= \frac{1}{\sqrt{2}}$ and $\bar v = 
\frac{-e^{i \theta}}{\sqrt{2}}$ and 
\bdm
\begin{pmatrix} u & v \lambda \\ -\bar{v} \lambda^{-1} & \bar{u} 
\end{pmatrix}
= \begin{pmatrix} 1 & 0 \\ e^{i \theta} \lambda^{-1} & 1 
\end{pmatrix} 
   \begin{pmatrix} \frac{1}{\sqrt{2}} & -\frac{1}{\sqrt{2}} 
e^{- i \theta} \lambda \\ 0 & \sqrt{2} \end{pmatrix}.
\edm
Pushing the  last factor into $\widehat{B}$ 
then gives the required factorization.  
In this case, $B\omega_1^{-1}$ is in $\mathcal{P}_1$, 
because it can be expressed as 
\bdm
\small\bbar e^{-i\theta/2}& 0 \\ 0 & e^{i\theta/2} \ebar \cdot
  \omega_1  \cdot
\bbar e^{i\theta/2} & 0 \\ 0 & e^{-i\theta/2} \ebar \widehat{B}.
\edm 
\end{proof}

%**********************************************
\section{The loop group formulation and DPW method for 
spacelike CMC surfaces in $\R^{2,1}$}
\label{section2}

The loop group formulation for CMC surfaces in ${\mathbb E}^3$, $\SSS^3$ and ${\mathbb H}^3$ evolved
from the work of Sym \cite{sym1985}, Pinkall 
and Sterling \cite{pinkallsterling},
and Bobenko \cite{bobenko, bobenko1994}.  
The Sym-Bobenko formula for CMC surfaces was given by 
Bobenko \cite{Bob:cmc, bobenko1994},
generalizing the formula for pseudo-spherical surfaces of Sym \cite{sym1985}.
The case that the ambient space is non-Riemannian is analogous, replacing
the compact Lie group $SU_2$ with the non-compact real form $SU_{1,1}$, 
as we show in this 
section.  

\subsection{The $SU_{1,1}$-frame}\label{section2-wayne}
The matrices
$\{ e_1, ~ e_2, ~ e_3 \} := \{\sigma_1, ~ - \sigma_2, ~i \sigma_3 \}$ form 
a basis for the Lie algebra $\mathfrak{g} = \mathfrak{su}_{1,1}$.
Identifying the Lorentzian 3-space $\R^{2,1}$ with $\mathfrak{g}$, with inner
product given by $\langle X,Y \rangle = \tfrac{1}{2} \text{trace} (XY)$,
we have 
\bdm
\langle e_1,e_1 \rangle = \langle e_2,e_2 \rangle = - \langle e_3,  e_3 \rangle = 1
\edm
 and 
$\langle \sigma_i,\sigma_j \rangle = 0$ for $i \neq j$.  

Let $\Sigma$ be a Riemann surface, and suppose $f: \Sigma \to \R^{2,1}$ is a spacelike immersion with mean
curvature $H \neq 0$. Choose conformal coordinates $z = x + iy$ and define
a function $u: \Sigma \to \real$ such that the metric is given by
\begin{equation}\label{eqn:dssquared}
\dd s^2 = 4e^{2u}(\dd x^2 + \dd y^2).
\end{equation}
We can define a frame $F: \Sigma \to SU_{1,1}$ by demanding that
\bdm
F e_1 F^{-1} = \frac{f_x}{|f_x|}, \hspace{1cm} 
F e_2 F^{-1} = \frac{f_y}{|f_y|}.
\edm
Assume coordinates for the target and domain are chosen such that 
$f_x(0) = |f_x(0)|e_1$ and $f_y(0) = |f_y(0)|e_2$,
so that $F(0) = I$. Then the frame $F$ is unique.
A choice of unit normal vector is given by
$N = F e_3 F^{-1}$.  
The Hopf differential is defined to be $Q \dd z^2$, where
\bdm
Q:= \langle N,f_{zz} \rangle = -\langle N_z,f_z \rangle.
\edm 

The Maurer-Cartan form, $\alpha$, for the frame $F$ is defined by
$\alpha := F^{-1} \dd F = U \dd z + V \dd \bar{z}$. 
\begin{lemma}
The connection coefficients $U := F^{-1}F_z$ and 
$V := F^{-1}F_{\bar{z}}$ are given by 
\begin{equation}\label{UhatandVhat}
U = \frac{1}{2} \begin{pmatrix} u_z & -2 i H e^u \\ 
                                   i e^{-u} Q & -u_z \end{pmatrix} 
, \hspace{1cm} V = \frac{1}{2} \begin{pmatrix} -u_{\bar z} & -i e^{-u} \bar Q \\ 
                            2 i H e^u & u_{\bar z} \end{pmatrix} 
\; . \end{equation}
The compatibility condition $\dd \alpha + \alpha \wedge \alpha = 0$
is equivalent to the pair of equations
\beqa 
\, && u_{z \bar z} - H^2 e^{2u}+\tfrac{1}{4}  |Q|^2 e^{-2u} = 0,
\label{compatibility1} \\
\, && Q_{\bar z} = 2 e^{2 u} H_z \; . \nonumber 
\eeqa
\end{lemma}
\begin{proof}
This is a straightforward computation, using
 $H = \frac{1}{8}e^{-2u}\langle f_{xx} + f_{yy}, N \rangle$, 
and the consequent $f_{zz} = 2u_z f_z - QN$, 
$f_{\bar{z} \bar{z}} = 2u_{\bar{z}}f_{\bar{z}} - \bar{Q}N$, 
$f_{z \bar{z}} = -2 H e^{2u}N$, 
in addition to 
 \begin{equation}
f_z = 2 e^u F \cdot \bbar  0 & 1 \\ 0 & 0 \ebar
\cdot F^{-1} \; , \hspace{1cm}
f_{\bar z} = 2 e^u F \cdot \bbar  0 & 0 \\ 1 & 0 \ebar
\cdot F^{-1} \; . \end{equation}  
\end{proof}

\subsection{The loop group formulation and the Sym-Bobenko formula} 
Now let us insert a parameter $\lambda$ into the $1$-form $\alpha$, defining 
the family
$\alpha^\lambda := U^\lambda \dd z + V^\lambda \dd \bar{z}$, where
\begin{equation}\label{withlambda}
U^\lambda = \frac{1}{2} \begin{pmatrix} u_z & -2 i H e^u \lambda^{-1} \\ 
                         i e^{-u} Q \lambda^{-1} & -u_z \end{pmatrix} 
, \hspace{.5cm} V^\lambda = 
\frac{1}{2} \begin{pmatrix} -u_{\bar z} & -i e^{-u} \bar Q \lambda \\ 
                     2 i H e^u \lambda & u_{\bar z} \end{pmatrix} 
\; . \end{equation}
It is simple to check the following fundamental fact:
\begin{proposition}
The $1$-form $\alpha^\lambda$ satisfies the Maurer-Cartan equation
\bdm
\dd \alpha^\lambda + \alpha^\lambda \wedge \alpha^\lambda = 0
\edm
 for all $\lambda \in \C \setminus \{ 0 \}$ 
if and only if the following two conditions both hold:
\begin{enumerate}
\item
$\dd \alpha^{1} + \alpha^{1} \wedge \alpha^{1} = 0$,
\item
the mean curvature $H$ is constant.
\end{enumerate}
\end{proposition}

Note that, comparing with (\ref{matrixforms}),
 $\alpha^\lambda$ is a 1-form with values in $Lie(\hh_\tau)$,
and is integrable for all $\lambda$. Hence it can be integrated to 
obtain a map $F: \Sigma \to \hh_\tau$.
%-----
\begin{definition} \label{extendedframedef}
The map $F: \Sigma \to \hh_\tau$ obtained by integrating the 
above  1-form $\alpha^\lambda$, with the initial condition $F(0) = I$,
 is called an \emph{extended frame}
for the CMC surface $f$.
\end{definition}
\begin{remark}
Such a frame $F$ is also an extended frame for a harmonic map, as,
for each $\lambda \in \SSS^1$, $F^\lambda$ projects to a harmonic
map into $SU_{1,1}/K$, where $K$ is the diagonal subgroup. We will
not be emphasizing that aspect in this article, however.
\end{remark}

When $H$ is a nonzero constant, the Sym-Bobenko formula, at $\lambda_0 \in \SSS^1$,
 is given by:
\beqa \label{symformula}
&&\hat{f}^{\lambda_0} = -\frac{1}{2H} 
\mathcal{S}(F) \Big|_{\lambda = \lambda_0},  \\
&& \sym(F) :=  F i\sigma_3 F^{-1} + 2 i \lambda 
\partial_\lambda F \cdot F^{-1} \; . 
\eeqa
%*******************

\begin{theorem} \label{symthm} $\,$ 
\begin{enumerate}
\item \label{symthm1} 
Given  a CMC $H$ surface, $f$, with  extended frame $F: \Sigma \to \hh_\tau$,
described above, the original surface $f$ is recovered,
up to a translation, from 
the Sym-Bobenko formula as $\hat{f}^1$. For other values of 
$\lambda \in \SSS^1$, $\hat{f}^\lambda$ is also a CMC $H$ surface in $\real^{2,1}$,
with 
Hopf differential given by $\lambda^{-2} Q$.
\item \label{symthm2}
Conversely, given a map $F: \Sigma \to \hh_\tau$ whose Maurer-Cartan form has 
coefficients of the form given by (\ref{withlambda}), the map $\hat{f}^\lambda$
obtained by the Sym-Bobenko formula is a CMC $H$ immersion into $\R^{2,1}$.
\item \label{symthm3}
If $D$ is any diagonal matrix, constant in $\lambda$, then
$\sym(FD) = \sym(F)$.
\end{enumerate}
\end{theorem}
\begin{proof}
For \ref{symthm1}, 
one computes that $\hat f^1_z=f_z$ and $\hat f^1_{\bar z}=f_{\bar z}$, 
so $f$ and $\hat f^1$ are the same surface up to translation. 
For other values of $\lambda$, see item \ref{symthm2}. 
To prove \ref{symthm2}, one computes $\hat{f}_z$ and $\hat{f}_{\bar{z}}$, 
and then the metric, the Hopf differential and the mean curvature.
Item \ref{symthm3} of the theorem is obvious.
\end{proof}
The family of CMC surfaces $\hat{f}^\lambda$ is called the 
\emph{associate family} for $f$. The invariance of the Sym-Bobenko
formula with respect to right multiplication by a diagonal matrix
is due to the fact that the surface is determined by its Gauss map,
given by the equivalence class of the frame in $SU_{1,1}/K$.

By direct computation using the first and second fundamental forms, we have: 
\begin{lemma} The surfaces 
\beqas
 \hat f^{1}_{||} &=& -\frac{1}{2H} \Ad_{\sigma_1}
\mathcal{S}(\Ad_{\sigma_1} F) \Big|_{\lambda = 1} \\
&=&
- \tfrac{1}{2 H} \left[ - F i \sigma_3 
F^{-1} + 2 i \lambda \partial_\lambda F \cdot F^{-1} 
\right]_{\lambda=1} \; , \\
\hat f^{1}_{K} &=& - \tfrac{1}{2 H} \left[ 
0 + 2 i \lambda \partial_\lambda F \cdot F^{-1} 
\right]_{\lambda=1} 
\eeqas
are the parallel CMC $-H$ surface and the parallel 
constant Gaussian curvature $-4H^2$ surfaces, respectively, to $\hat f^1$.  
\end{lemma}

%-------------------------
\subsection{Extending the construction to $G$}
In the formulation above we used the group $SU_{1,1}$, 
but we can use the bigger group $G$ 
instead, and allow the extended frame to take values in  
$\hh = \hh_\tau \sqcup \Psi(i \sigma_1) \cdot \hh_\tau$.  
If we integrate the 1-form $\alpha^\lambda$ above, with the initial condition
$\hat{F}(0) = \Psi(i\sigma_1)$ instead of the identity, we obtain a frame, $\hat{F} = \Psi(i\sigma_1) F$,
 with values in
$\Psi(i\sigma_1) \cdot \hh_\tau$.  But 
$\sym(\Psi(i \sigma_1) F) = - \textup{Ad}_{\sigma_1} \sym(F) + \textup{translation}$,
and the effect of 
$- \textup{Ad}_{\sigma_1}$  on the surface is just 
an isometry of $\R^{2,1}$, and so a CMC surface is obtained.
Similarly, it is clear that we can replace $\hh_\tau$ with $\hh$ in the converse part of Theorem \ref{symthm}.

%------------------------------------------------------------------
\subsection{The DPW method for $\R^{2,1}$}\label{section3-wayne}

Here we give the  holomorphic 
representation of the extended frames constructed above.
To see how it works in practice, consult the examples below,
in Section \ref{examplessect}.

On a simply-connected 
Riemann surface $\Sigma$ with local coordinate $z=x+i y$, we define a 
{\em holomorphic potential} as an $\mathfrak{sl}_2 \C$-valued 
$\lambda$-dependent $1$-form 
\[ \xi = A(z,\lambda) dz = \begin{pmatrix} 
\sum_{j=0}^\infty c_{2j} \lambda^{2j} & 
\sum_{j=0}^\infty a_{2j-1} \lambda^{2j-1} \\ 
\sum_{j=0}^\infty b_{2j-1} \lambda^{2j-1} & 
-\sum_{j=0}^\infty c_{2j} \lambda^{2j} \end{pmatrix} dz \; , \]  where 
the $a_j dz, b_j dz, c_j dz$ are all holomorphic 
$1$-forms defined on $\Sigma$, and 
$a_{-1}$ is never zero.  

Choose a solution $\phi : \Sigma \to 
\hh^\C$ of $d\phi = \phi \xi$, and 
$G$-Iwasawa split $\phi=F B$ with $F: 
\Sigma \to \hh$ and 
$B :\Sigma \to \uhat$ whenever 
$\phi \in \mathcal{B}_{1,1}$.  Expanding 
\bdm
B = 
\bbar \rho & 0 \\ 0 & \rho^{-1} \ebar
+ \mathcal{O}(\lambda), \hspace{1cm} \rho(z,\bar{z}) \in \R^+,
\edm 
and, noting that 
\bdm
F^{-1} dF = BAB^{-1} dz - dB \cdot B^{-1}
\edm
 and 
$\tau({F^{-1} dF}) =  F^{-1} dF$, one deduces that 
\beqas
\begin{aligned}
&F^{-1} dF =
 \mathcal{A}_1 dz + \mathcal{A}_2 dz + 
\tau(\mathcal{A}_2) d\bar z +  \tau(\mathcal{A}_1) 
d\bar z \; ,\\
&\mathcal{A}_1={\bbar  0 & \lambda^{-1} \rho^2 a_{-1}\\
    \lambda^{-1} \rho^{-2} b_{-1} & 0 \ebar}  , \hspace{1cm} 
    \mathcal{A}_2= \bbar   \tfrac{\rho_z}{\rho} & 0 \\ 
       0 & -\tfrac{\rho_z}{\rho} \ebar . 
\end{aligned}
\eeqas

Take any nonzero real constant $H$.  
Substituting $w = \tfrac{i}{H} \int a_{-1} dz$,
$Q=-2 H \tfrac{b_{-1}}{a_{-1}}$ and $\rho = e^{u/2}$, we have 
$F^{-1} dF = U^\lambda dw+V^\lambda d\bar w$ for 
$U^\lambda(w)$, $V^\lambda(w)$ as 
in Section \ref{section2}.  By Theorem \ref{symthm},  $F$ is an extended frame
for a family of spacelike CMC $H$ immersions. 
\begin{remark} 
The invariance of the Sym-Bobenko formula, pointed out
in Theorem \ref{symthm}, shows that we did not need to choose the
unique $F \in \hh$  given by the normalization
 $B \in \uhat$ in
our splitting of $\phi$ above,
because the freedom for $F$ (Theorem \ref{section1-thm:SU11Iwa}) is postmultiplication by $\hh^0$,
which consists of diagonal matrices. 
The normalized choice of $B$, however, will be used sometimes, as
it captures some information 
about the metric of the surface in terms of $\rho$.

We also point out that allowing $a_{-1}$ to have zeros will result in
a surface with branch points at these zeros.
\end{remark}

We have proved one direction of the following theorem, which gives a 
holomorphic representation 
for all non-maximal CMC spacelike surfaces in $\R^{2,1}$.   
 In the converse statement, the main issue is that we do not 
assume $\Sigma$ is simply-connected, which can be important for
applications: see, for example \cite{DorH:cyl}, \cite{DH98}.

\begin{theorem} \label{dpwthm}
 (Holomorphic representation for spacelike 
CMC surfaces in $\R^{2,1}$)
Let 
\bdm
\xi = \sum_{i=-1}^\infty A_i \lambda^i \dd z 
~~ \in ~ Lie(\hh^\C) \otimes \Omega ^{1,0} (\Sigma)
\edm
be a holomorphic 1-form over a simply-connected Riemann surface 
$\Sigma$, with
\bdm
 a_{-1} \neq 0,
\edm
on $\Sigma$, where $A_{-1} = {\small{\bbar 0 & a_{-1} \\ b_{-1} & 0 \ebar}}$. Let 
$\phi :\Sigma \to \hh^\C$ be 
a solution of 
\bdm
\phi^{-1} d\phi=\xi.
\edm  
Define the open set 
$\Sigma^\circ := \phi ^{-1} (\mathcal{B}_{1,1})$,
 and take any  $G$-Iwasawa splitting on $\Sigma^\circ$: 
\beq \label{thmsplit}
\phi = F B, \hspace{1.5cm} F \in \hh, \hspace{.2cm} B \in \hh^\C_+.
\eeq
  Then for any $\lambda_0 \in \SSS^1$, the map
$f^{\lambda_0} := \hat f^{\lambda_0}: \Sigma^\circ \to \mathbb{R}^{2,1}$, given by
the Sym-Bobenko formula \eqref{symformula}, is a conformal CMC $H$ immersion,
and is independent of the choice of $F$ in (\ref{thmsplit}).

Conversely, let $\Sigma$ be a non-compact Riemann surface.  Then 
any non-maximal conformal CMC spacelike immersion 
from $\Sigma $ into $\R^{2,1}$ can be constructed in this manner, 
using a holomorphic potential $\xi$ that is well defined on $\Sigma$.   
\end{theorem}

\begin{proof}
The only point remaining to prove is the converse statement. This
follows from our construction of the extended frame associated to
any such surface, together with 
the argument in \cite{DorPW} (Lemma 4.11 and the Appendix)
given  for the case that $\Sigma$ is contractible.
However, the latter argument is also valid if $\Sigma$ is any
 non-compact Riemann surface: the global statement only 
depends on the generalization of Grauert's Theorem given in \cite{bungart},
that any holomorphic vector bundle over a Stein manifold (such as a non-compact
Riemann surface, see \cite{grauertremmert} Section 5.1.5)
with fibers in a Banach space, is trivial.
\end{proof}

\begin{remark}  \label{metricremark}
We also showed above that if we normalize the factors 
in (\ref{thmsplit}) so that $B \in \uhat$,
and define the function $\rho: \Sigma^\circ \to \real$ by $B|_{\lambda=0} = \textup{diag}(\rho, \rho^{-1})$, then there exist conformal coordinates 
$\tilde z= \tilde x+i \tilde y$ on $\Sigma$
such that the induced metric for  $f^1$ is given by
\bdm
\dd s^2 = 4 \rho^4 (\dd \tilde x^2 + \dd \tilde y^2),
\edm
and the Hopf differential is given by $Q \dd \tilde z^2$, where $Q = -2H\frac{b_{-1}}{a_{-1}}$.
\end{remark}
%---------------------------
\subsection{Preliminary examples} \label{examplessect}
We conclude this section with three  examples: 

\begin{example}\label{exa:cylinders} A cylinder over a hyperbola in $\R^{2,1}$.  
Let 
\bdm
\xi = {\small{\bbar  0 & \lambda^{-1} dz \\ \lambda^{-1} dz & 0\ebar}},
\edm
 on $\Sigma = \C$.  Then one solution $\phi$ of $d\phi = \phi \xi$ is 
\bdm
\phi = 
\exp \Big\{ \bbar 0 &  z \lambda^{-1} \\ z \lambda^{-1} &0 \ebar \Big\}, 
\edm
which has the Iwasawa splitting
$\phi = F \cdot B$,
where
\bdm
F = \exp \Big\{ 
\bbar  0 & z \lambda^{-1} + \bar z \lambda \\ z \lambda^{-1} + \bar z \lambda & 0\ebar \Big\}, \hspace{1cm} 
B = \exp  \Big\{ \bbar 0 & -\bar z\lambda \\ -\bar z\lambda & 0 \ebar \Big\} ,
\edm
take values in $\hh$ and $\uhat$ respectively.  
The Sym-Bobenko formula $\hat f^1$ gives the surface 
\bdm
\frac{-1}{2H} \cdot 
[ 4 y, \, -\sinh(4 x), \,\cosh(4 x) ],
\edm
in $\R^{2,1} = \{ [x_1,x_2,x_0] := 
x_1e_1+x_2e_2+x_0e_3 \}$.
  The image is the set 
\bdm
\{ [x_1,x_2,x_0]  ~|~ x_0^2-x_2^2 = \tfrac{1}{4H^2} \},
\edm
 which is a 
 cylinder over a hyperbola.  
\end{example}
%------------

\begin{example}\label{exa:spheres} The hyperboloid of two sheets.   
Let 
\bdm
\xi = \bbar 0 & \lambda^{-1}\\ 0&0 \ebar dz,
\edm
 on $\Sigma=\C$.  
Then one solution of $d \phi = \phi \xi$ is 
\bdm
\phi = \bbar  1 & z \lambda^{-1} \\ 0 &1 \ebar,
\edm
which takes values in  
$\mathcal{B}_{1,1}$ for 
$|z| \neq 1$. For these values of $z$, the $G$-Iwasawa splitting is
 $\phi
= F \cdot B$ with $F : \Sigma \setminus 
\SSS^1 \to \hh$ and 
$B:  \Sigma \setminus \SSS^1 
\to \uhat$, where 
\beqas
 F &=& \frac{1}{\sqrt{\varepsilon (1-|z|^2)}} 
\begin{pmatrix} \varepsilon & z \lambda^{-1} \\ 
\varepsilon \bar z \lambda & 1 \end{pmatrix} ,\\
B &=& \frac{1}{\sqrt{\varepsilon (1-|z|^2)}}
 \bbar  1 & 0 \\ -\varepsilon \bar{z} \lambda & \varepsilon(1-z \bar{z}) \ebar,
 \hspace{1cm}
\varepsilon  = \text{sign}(1-|z|^2) \; . 
\eeqas
Then the Sym-Bobenko formula gives 
\bdm
\hat{f}^1(z) = \frac{1}{H (x^2+y^2-1)} \cdot [ 2y, \, -2x, \, (1+3x^2+3y^2)/2],
\edm 
whose image is the two-sheeted hyperboloid
 $\{ x_1^2 + x_2^2 -(x_0-\frac{1}{2H})^2 = -\frac{1}{H^2} \}$,
that is, two copies of a hyperbolic plane  of constant 
curvature $-H^2$.  
For this example, we are in a small cell precisely when $|z|=1$.  In this 
case, we can write $\phi$ as a product of a loop in $\hh_\tau$ 
times $\omega_2$ times a loop in $\hh^\C_+$, as follows: 
\[ \begin{pmatrix}
1 & z \lambda^{-1} \\ 0 & 1 
\end{pmatrix} = 
\begin{pmatrix}
p \sqrt{z} & \lambda^{-1} q \sqrt{z} \\ 
\lambda q \sqrt{z}^{-1} & p \sqrt{z}^{-1} 
\end{pmatrix} 
\cdot \omega_2 \cdot 
\begin{pmatrix}
(p+q) \sqrt{z}^{-1} & 0 \\ 
- \lambda q \sqrt{z}^{-1} & (p-q) \sqrt{z}
\end{pmatrix} \; , 
\] where $p^2-q^2=1$ and $p$, $q \in \real$.  Hence $\phi \in 
\mathcal{P}_{2}$ for $|z|=1$.
\end{example}

\begin{example}\label{beaksexample}
The first two examples were especially simple, so that we were able to perform the
 Iwasawa splitting   explicitly. This is not possible, in general.
  However, it can always be approximated numerically, using, for example,
 the program XLab \cite{xlab},
 and images of the surface corresponding to an arbitrary potential
$\xi$ can be produced. For example, taking the potential 
$\xi = \lambda^{-1} \cdot {\small{\bbar 0& 1 \\  \, 100 z & 0 \ebar}} dz$,
and integrating with the initial condition $\phi(0)=\omega_1$,
we obtain, numerically, a surface with  a singularity that
appears to have the topology of a Shcherbak surface \cite{shcherbak}
 singularity at $z=0$. The Shcherbak surface singularity is of the form
 $(u, ~v^3 +uv^2, ~12 v^5 + 10u v^4)$.
 The singularity from our construction is displayed
 in Figure \ref{fg:7}. Since $\phi(0) = \omega_1$, this singularity 
is arising when $\phi$ takes values in $\mathcal{P}_1$. 
 \end{example}

%******************************************************************
%##################################################################

\begin{figure}[here]
\begin{center}
\includegraphics[width=.7\linewidth]{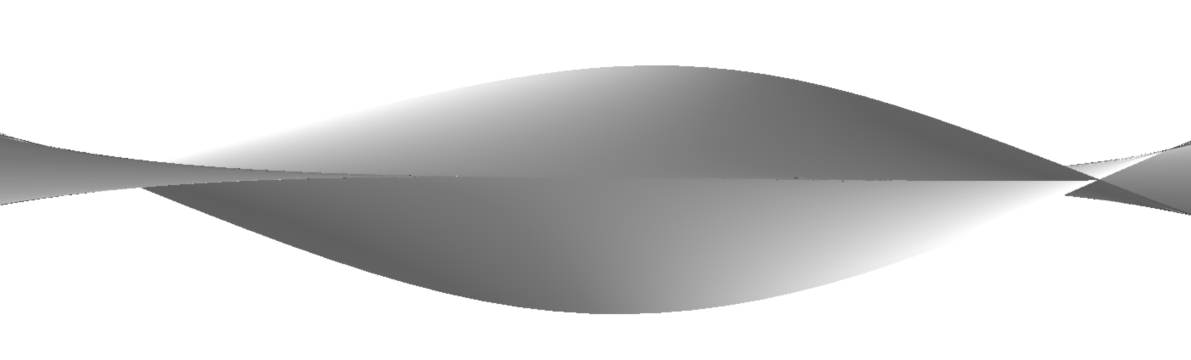}
%\vspace{-1cm}
\end{center}
\caption{The singularity appearing in  Example \ref{beaksexample}}
\label{fg:7}
\end{figure}

%******************************************************************
%##################################################################

%*************************************************************
\section{Behavior of  the Sym-Bobenko formula on the boundary of the big cell}
\label{brander-section}

We saw in Example \ref{exa:spheres} an instance of a surface which
blows up as the boundary of the big cell is approached. 
On the other hand, in Example \ref{beaksexample}, we have
a case where finite singularities occur. We now want
to examine what behavior can be expected in general.

Let $\phi: \Sigma \to \hh^\C$ be a holomorphic map in accordance with
the construction of 
Theorem \ref{dpwthm}, and $\Sigma^\circ := \phi^{-1}(\mathcal{B}_{1,1})$.
We also assume that $\phi$ maps at least one point into $\mathcal{B}_{1,1}$,
so that $\Sigma^\circ$ is not empty.
Set 
\bdm
C:= \Sigma \setminus \Sigma^\circ = \bigcup_{j=1}^\infty \phi^{-1} (\mathcal{P}_j).
\edm

%************
\begin{theorem} \label{summarythm1}
Let $\phi$ be as above, and assume that $\Sigma$ is simply connected.
 Then $\Sigma^\circ$ is open and dense in $\Sigma$. 
More precisely, its complement, the set $C$, 
is locally given as the zero set 
 of a non-constant real analytic function from some open set $W \subset \Sigma$ 
 to $\C$.
\end{theorem}
%-----------------------------

\begin{proof}
This follows from item  \ref{mainthm-(3)} of Theorem \ref{section1-thm:SU11Iwa}:
the union of the small cells is given as the zero set of a real analytic section
$s$ of a real analytic line bundle on $\hh^\C$ (see the proof of  Theorem
\ref{thm:SU11Iwa}). Thus $C$ is given as the zero set of $\phi^* s$, which 
is also a real analytic section of a real analytic line bundle.  Since we assume
that  the complement
of $C$ contains at least one point, it follows that this set is open and dense. 
\end{proof}

For the first two small cells, for which the analysis is the least complicated,
 we will prove more
specific information: set
\bdm
C_1 := \phi^{-1} (\mathcal{P}_1),
\hspace{1cm}
C_2 := \phi^{-1} (\mathcal{P}_2).
\edm

%--------------------
\begin{theorem} \label{summarythm}  Let $\phi$ be as given in
Theorem \ref{summarythm1}. Then:
\begin{enumerate}

\item  \label{summary1}
The sets $\Sigma^\circ \cup C_1$ and $\Sigma^\circ \cup C_2$ 
are both {open} subsets of $\Sigma$.
The sets $C_i$ are each locally given as the zero set 
 of a non-constant real analytic function $\R^2 \to \R$.   
\item \label{summary2}
All components of the matrix $F$ obtained by Theorem \ref{dpwthm} on $\Sigma^\circ$,
and evaluated at $\lambda_0 \in \SSS^1$, blow up as $z$ approaches a point $z_0$ in
either $C_1$ or $C_2$. In the limit, the unit normal vector $N$, to the corresponding
surface, becomes asymptotically lightlike, i.e. its length in the Euclidean
$\real^3$ metric approaches infinity.
\item \label{summary3}
The surface $f^{\lambda_0}$ obtained from Theorem \ref{dpwthm}
 extends to a real analytic map 
$\Sigma^\circ \cup C_1 \to \R^{2,1}$, but 
is not immersed at points $z_0  \in C_1$.
\item \label{summary4}
The surface $f^{\lambda_0}$ diverges to  $\infty$ as
$z \to  z_0 \in C_2$. Moreover, the induced metric on the surface blows
up as such a point in the coordinate domain is approached.
\end{enumerate}

\end{theorem}

\begin{proof}
Item \ref{summary1}: 
For the open condition, it is enough to show that if $z_0 \in \Sigma^\circ \cup C_i$,
then there is a neighborhood of $z_0$ also contained in this set.  
Let $z_0 \in \Sigma^\circ \cup C_1$. Now $\Sigma^\circ$ is open, so take $z_0 \in C_1$.
It easy to see 
 that, in the following argument, no generality is
lost by assuming that $\phi(z_0) = \omega_1^{-1}$.
We can express $\phi$ as
\bdm
\phi = \hat{\phi} \omega_1^{-1},
\edm
where $\hat{\phi} := \phi \omega_1$. Since $\hat{\phi}(z_0) = I$, the identity,
$\hat{\phi}(z)$ is in the big cell in a neighborhood of $z_0$, and therefore can
locally be expressed as
\bdm
\hat{\phi} = F B, \hspace{1cm} F: \Sigma \to \hh, \hspace{.5cm} B: \Sigma \to \hh^\C_+.
\edm
So $\phi = F B \omega_1^{-1}$, and, denoting the components of $B$ as in  
Lemma \ref{switchlemma}, we have that $\phi(z)$ is in
$\mathcal{P}_1$ precisely when 
\bdm
g(z) := |b_1(z) - a_0(z)| |a_0(z)| -1 = 0,
\edm
and is in the big cell for other values of this function. Note that $g$ cannot
be constant, because, by Theorem \ref{summarythm1}, $z_0$ is
 a boundary point of $\Sigma^\circ$.
The case $z_0 \in \Sigma^\circ \cup C_2$ is analogous, and
the claim follows.

Items \ref{summary2}-\ref{summary4} are proved below as Corollaries 
 \ref{sumcor2},  \ref{sumcor3} and \ref{sumcor4} respectively.
\end{proof}
%---------------------
\begin{remark}
Noting that $\Ad_{\sigma_1} \omega_{2k-1} = \omega_{2k}$, 
and that the parallel surface is obtained by applying the Sym-Bobenko
formula to $\Ad_{\sigma_1} F$, the analogue of Theorem \ref{summarythm}
applies to the parallel surface, switching $\mathcal{P}_1$ and $\mathcal{P}_2$.
\end{remark}

%------------------------*************************
\subsection{Behavior of the $\hh$ and $\hh^\C_+$ factors approaching the first two small cells}
We can use Lemma \ref{switchlemma} to show that the matrix $F$, in an 
$SU_{1,1}$ Iwasawa
factorization $\phi = F B$, blows up as $\phi$ approaches either of the first 
two small cells. Note that all such discussions take place for $\lambda \in \SSS^1$,
so that, for example, if $a$ is a function of $\lambda$, then $a^* = \bar{a}$.
%---------------------
\begin{proposition} \label{blowupprop}
Let $\phi_n$ be a sequence in $\mathcal{B}_{1,1}$, with 
 $\lim_{n \to \infty} \phi_n = \phi_0 \in \mathcal{P}_m$, for $m = 1$ or $2$.
Let $\phi_n = F_n B_n$ be an
$SU_{1,1}$ Iwasawa decomposition of $\phi_n$, 
with $F_n \in \hh$, 
$B_n \in \hh^\C_+$. Then:
\begin{enumerate}
\item
Writing $F_n$ as 
\bdm
F_n = \begin{pmatrix} x_n & y_n \\ \pm y^*_n & 
\pm x^*_n \end{pmatrix} \; , 
\edm
we have 
$\lim_{n \to \infty} |x_n| = \lim_{n \to \infty} |y_n| = \infty$, 
for all $\lambda \in \SSS^1$.\\
\item
Writing the constant term of $B_n$ as 
\bdm
B_n \big |_{\lambda =0} = \bbar \rho_n & 0 \\ 0 &  \rho_n^{-1} \ebar,
\edm
if $m=1$ then $\lim_{n \to \infty} |\rho_n| = 0$, and if 
$m=2$ then $\lim_{n \to \infty} |\rho_n| = \infty$.
\end{enumerate} 
\end{proposition}
\vspace{.5cm}
%-------
\begin{proof}
\textbf{Item (1):} We give the proof for $m=1$. The case $m=2$ can be proved 
in the same way, or simply
obtained from the first case by applying $\Ad_{\sigma_1}$. According to
Theorem \ref{section1-thm:SU11Iwa}, we can write
\bdm
\phi_0 = F_0 \omega_1 B_0,
\edm
with $F_0 \in \hh_\tau$ and
  $B_0 \in \hh^\C_+$.  Expressing $\phi_n$ as 
\[
\phi_n = \hat{\phi}_n \omega_1 B_0 \; , \hspace{1cm}
\hat{\phi}_n := \phi_n B_0^{-1} \omega_1^{-1} \; , 
\]
we have 
$\lim_{n \to \infty}\hat{\phi}_n = F_0$, so $\hat{\phi}_n \in 
\mathcal{B}_{1,1}$ for sufficiently large $n$, because 
$\mathcal{B}_{1,1}$ is open.  Thus, for large $n$,
we have the factorization $\hat{\phi}_n = \hat{F}_n \hat{B}_n$, and the factors
can be chosen to satisfy
$\hat{F}_n \to F_0$ and $\hat{B}_n \to I$, as $n \to \infty$.
Using Lemma \ref{switchlemma}, with $\lambda$ replaced by 
$-\lambda$, we have the expression
\bdm
\phi_n = \hat{F}_n \hat{B}_n \omega_1 B_0 = \hat{F}_n X_n \tilde{B}_n B_0 \; 
, \hspace{.7cm} \tilde{B}_n \in \hh^\C_+ \; . 
\edm
Since by assumption 
$\phi_n \in \mathcal{B}_{1,1}$ for all $n$, the factor 
$\hat{B}_n \omega_1$ is also, and $X_n$ is always
a matrix of the form $k_1$ or $k_2$, that is 
\bdm
X_n = \begin{pmatrix} u_n & v_n \lambda \\ 
\pm \bar{v}_n \lambda^{-1} & \pm \bar{u}_n \end{pmatrix}, 
\edm
with $u_n$ and $v_n$ constant in $\lambda$.
We also have from Lemma \ref{switchlemma}, that $|u_n|/|v_n| =  
|\hat{b}_{1,n} - \hat{a}_{0,n}||\hat{a}_{0,n}|$,
where $\hat{b}_{1,n} \to 0$ and $\hat{a}_{0,n} \to 1$, as 
$n \to \infty$, 
because $\hat{B}_n \to I$.
Hence 
$\lim_{n \to \infty} \frac{|u_n|}{|v_n|} = 1$.
Combined with the condition $|u_n|^2 - |v_n|^2 = \pm 1$, this implies that 
$\lim_{n \to \infty} |u_n| = \lim_{n \to \infty} |v_n| = \infty$, and
\bdm
\lim_{n \to \infty} ||X_n|| = \infty,
\edm
 where $|| \cdot ||$ is some suitable
matrix norm.  
Now the uniqueness of the Iwasawa splitting $\phi_n = F_n B_n$ says that 
\bdm
F_n = \hat{F}_n X_n D_n,
\edm
where $D_n = \textup{diag}(e^{i \theta_n}, e^{-i \theta_n})$ 
for some $\theta_n \in \R$.  Then we have
\bdm
||X_n|| = ||\hat{F}^{-1}_n F_n || \leq ||\hat{F}^{-1}_n ||  \, || F_n||,
\edm
and so $\lim_{n \to \infty} ||\hat{F}^{-1}_n ||  \, || F_n|| = \infty$ also.
But $||\hat{F}^{-1}_n|| \to ||F_0||$, which is finite, 
and so we have $|| F_n || \to \infty$.
Because the components of $F_n$ satisfy $|x_n|^2 - |y_n|^2 = \pm 1$,
the result follows. \\
\noindent \textbf{Item (2):}
For the case $m=1$, proceeding as above, we have
$\phi_n =  \hat{F}_n X_n \tilde{B}_n B_0$, where $X_n \tilde{B}_n = \hat{B}_n \omega_1$,
and $\hat{B}_n \to I$. Up to some constant factor coming from $B_0$, 
the quantity $\rho_n^{-1}$ is given by the constant term  of the
matrix component $[\tilde{B}_n]_{22}$, for which
we have an explicit expression in (\ref{explicitfact}), that is:
\bdm
\rho_n^{-1} =  -\varepsilon \hat b _{n,1} \, \bar{v}_n + u_n \, \hat{d}_{n,0},
\hspace{.5cm} \textup{where } \hspace{.2cm} \hat{B}_n = \bbar  \sum_{i=0}^\infty \hat a_{n,i} \lambda ^i &
 \sum_{i=1}^\infty \hat b_{n,i} \lambda ^i \\
  \sum_{i=1}^\infty \hat c_{n,i} \lambda ^i &
   \sum_{i=0}^\infty \hat d_{n,i} \lambda ^i \ebar.
\edm
Now the facts that $\hat{B}_n \to I$ and $u_n \bar u_n - v_n \bar v_n = \varepsilon$,
so that
\beqas
b _{n,1} \to 0, \hspace{1cm} \hat{d}_{n,0} \to 1,\\
\frac{|u_n|}{|v_n|} = |\hat{b}_{n,1} - \hat{a}_{n,0}||\hat{a}_{n,0}| \to 1,
\hspace{1cm}  |u_n| \to \infty,
\eeqas
imply that $|\rho_n^{-1}| \to \infty$,
 which is what we needed to show.  The case $m=2$ is obtained by applying 
 $\Ad_{\sigma_1}$, which switches $\rho$ and $\rho^{-1}$.
\end{proof}

\begin{corollary} \label{sumcor2}
Proof of item \ref{summary2} of Theorem \ref{summarythm}.
\end{corollary}
\begin{proof} We just saw that all components of $F$ blow up as $\phi$ 
approaches $\mathcal{P}_1$ or $\mathcal{P}_2$.  
Taking $F = {\small{\bbar a & b \\ \pm  b^* & \pm  a^* \ebar}}$, 
 Proposition \ref{blowupprop} says $|a| \to \infty$ and $|b| \to \infty$. 
The unit normal vector is
given by
\bdm
F i\sigma_3 F^{-1} = i \cdot \bbar \pm(aa^* + bb^*) & -2ab\\ 2a^*b^* & \mp(bb^* +aa^*) \ebar.
\edm
The $e_3$ component, $\pm(aa^* + bb^*)$, approaches $\infty$. Since $N$ is a unit vector,
the only way this can happen is for the vector to become asymptotically lightlike.
\end{proof}

%****************************************************
\subsection{Extending the Sym-Bobenko formula to the first small cell}
To show that the surface extends analytically to $C_1 = \phi^{-1}(\mathcal{P}_1)$,
we think of the Sym-Bobenko formula as a map from $\hh^\C$, instead of $\hh$,
by composing it with the projection onto $\hh$. This is necessary because we
showed that the $\hh$ factor blows up as we approach $\mathcal{P}_1$.

Recall the function $\sym$ in \eqref{symformula} used for 
the Sym-Bobenko formula.  
Note that if $F \in \hh$ then either $F$ or $iF$ is an element of 
$\Lambda \hat{G}_\sigma \subset \hh^\C$, where 
$\hat{G} = U_{1,1}$.  The Lie algebra of $\hat{G}$ is just 
$\mathfrak{g} = \mathfrak{su}_{1,1}$ and we can conclude 
that $F i\sigma_3 F^{-1}$ and $i \lambda 
\partial_\lambda F \cdot F^{-1}$ are loops 
in $Lie(\hh)$.  Thus $\sym$ is a real analytic map
from $\hh$ to $Lie(\hh)$.
Define 
\bdm
\mathcal{K} := \{ k \in \hh ~|~ \mathcal{S}(k) = i \sigma_3 \}.
\edm
\begin{lemma}  \label{symlemma}
$\mathcal{K}$ is a subgroup of $\hh$. Moreover, $\mathcal{K}$
consists precisely of the elements $k \in \hh$ such that  
\beq \label{keqn}
\mathcal{S}(Fk) = \mathcal{S}(F),
\eeq
for any $F \in  \hh$. 
\end{lemma}

\begin{proof}
Both statements follow from the easily verified formula
\bdm
\mathcal{S}(xy) = x \mathcal{S}(y) x^{-1} + 2 i\lambda 
\partial_\lambda x \cdot x^{-1}. 
\edm
Hence it is straightforward to show that $\mathcal{K}$  
is a group, and any element $k \in \mathcal{K}$
satisfies (\ref{keqn}) for any $F$. Conversely,
if $k$ is an element such that
(\ref{keqn}) holds for all $F$, in particular for $F = I$, then 
$\mathcal{S}(k) = \mathcal{S}(I) = i \sigma_3$, so $k \in \mathcal{K}$.
\end{proof}
Now $\hh^0$ consists of constant diagonal matrices, which are in $\mathcal{K}$, so an immediate
corollary of this lemma (see also Theorem \ref{symthm}) is
\begin{lemma} \label{extendtoquotient}
The function $\sym$ is a well-defined real analytic map
$\hh/\hh^0 \to Lie(\hh)$.
\end{lemma}

On the big cell, $\mathcal{B}_{1,1}$, we can define an extended Sym-formula 
$\stilde : \mathcal{B}_{1,1} \to Lie(\hh)$, 
by the composition
\beq \label{sym3}
\widetilde{\mathcal{S}}(\phi) :=  \mathcal{S} (\pi (\phi)),
\eeq
where $\pi$ is the projection to  $\hh/\hh^0$ given by the 
$SU_{1,1}$ Iwasawa splitting, described in Corollary \ref{cor1}.
  It is a real analytic function
on $\mathcal{B}_{1,1}$, since it is a composition of two such functions. 
In spite of the conclusion of 
Proposition \ref{blowupprop}, we now show that this 
function extends 
to the first small cell $\mathcal{P}_1$. The critical point 
in the following argument is
the easily verified fact that the matrices $k_i$ given in 
Lemma \ref{switchlemma}
are  elements of  $\mathcal{K}$. The argument does not 
apply to the second small
cell, because the corresponding matrices $\Ad_{\sigma_1} k_i$ are \emph{not}
elements of $\mathcal{K}$.

%==========================================================
\begin{theorem} \label{extensionthm}
The function $\stilde$ extends to a real analytic function 
$\mathcal{B}_{1,1} \sqcup \mathcal{P}_1 \to Lie(\hh)$.
\end{theorem}
\begin{proof}
Let $\phi_0$ be an element of $\mathcal{B}_{1,1} \sqcup \mathcal{P}_1$.
If $\phi_0 \in \mathcal{B}_{1,1}$ define $\stilde (\phi_0)$ 
by (\ref{sym3}), and
this is well defined and analytic in a neighborhood of $\phi_0$. 
If $\phi_0 \in \mathcal{P}_1$,
we have a factorization 
\beq \label{ssdecomp}
\phi_0 = F_0 \omega_1 B_0,
\eeq
  given by (\ref{globaldecomp}).
Then $\phi_0 B_0^{-1} \omega_1^{-1}$ is in $\mathcal{B}_{1,1}$, which is an
open set. Hence we can define, for $\phi$ in some neighborhood $\mathcal{W}_0$
of $\phi_0$, a new element 
\bdm
\hat{\phi} := \phi B_0^{-1} \omega_1^{-1},
\edm
and $\hat{\phi}$ is in $\mathcal{B}_{1,1}$ for all $\phi \in \mathcal{W}_0$.
Now we define, for $\phi \in \mathcal{W}_0$,
\beq \label{sym4}
\hat{\mathcal{S}} (\phi) := \stilde(\hat{\phi}).
\eeq
We need to check that this 
is well defined (because $B_0$ is not unique in (\ref{ssdecomp}))
and also that \eqref{sym3} and \eqref{sym4} coincide on 
$\mathcal{W}_0 \cap \mathcal{B}_{1,1}$.  
To prove both of these points it is enough to show just the second one, 
because $\mathcal{W}_0 \cap \mathcal{B}_{1,1}$,
is dense in $\mathcal{W}_0$ and because \eqref{sym4} is 
defined and continuous on the whole of $\mathcal{W}_0$. Now on 
$\mathcal{W}_0 \cap \mathcal{B}_{1,1}$, we have the Iwasawa factorization
$\phi = F B$, so 
\bdm
\hat{\phi} = F B B_0^{-1} \omega_1^{-1}.
\edm
Since we know this is in the big cell, we can express this, 
by Lemma \ref{switchlemma},
 as
\bdm
\hat{\phi} = F k B^\prime,
\edm
where $k$ is of the form 
$\begin{pmatrix} u & v \lambda \\ \pm \bar{v} \lambda^{-1} & \pm \bar{u} 
\end{pmatrix}$, and $B^\prime$ is in $\hh^C_+$.
Now $Fk \in \hh$, so, by definition,
$\stilde(\hat{\phi}) = \mathcal{S}(Fk)$. But  $k \in \mathcal{K}$, 
so, in fact, $\hat{\mathcal{S}}(\phi):=
\stilde(\hat{\phi}) = \mathcal{S}(F) = \stilde(\phi)$.
\end{proof}

\pagebreak
%----------------------------------------
\begin{corollary} \label{sumcor3}
Proof of item \ref{summary3} of Theorem \ref{summarythm}.
\end{corollary}
\begin{proof}
We just showed that the surface obtained by the Sym-Bobenko formula extends
to a real analytic map from $\Sigma^\circ \cup C_1$.  
To prove that the surface is not immersed at $z_0 \in C_1$, suppose the
contrary: that is, there is an open set $W$ containing $z_0$ such 
that ${f}^{\lambda_0}: W \to \R^{2,1}$ is an immersion.
Let $\dd \hat s ^2$ denote the induced 
metric.  From Remark \ref{metricremark}, this metric is given on 
the open dense set $\Sigma^\circ$ by 
the expression $4 \rho^4 (\dd x^2 + \dd y^2)$.
The 1-form $\dd x^2 + \dd y^2$ is well defined on $\Sigma$, but, by
item (2) of Proposition \ref{blowupprop}, the 
function $\rho^4$ approaches $0$ as $z$ approaches $z_0$.
 Therefore the induced metric is zero at this point. This is a contradiction, because a conformally
 immersed surface in $\R^{2,1}$ cannot be null at a point.
\end{proof}

%*****************************************************:
\subsection{The behavior of $\stilde$ when approaching other small cells}
 The function $\stilde$ does not extend continuously to any of the other
small cells. To see this, consider the functions $\psi^m_z$ 
and $F_z^m$ given in Remark \ref{psidefremark}.
  On the big cell, we have 
\beqas
\stilde(\psi^m_z) &=& \mathcal{S}(F^m_z) 
 = i\sigma_3 +
   \frac{2i(m-1)}{1-z\bar{z}} \begin{pmatrix} - z \bar{z} & \bar{z}\lambda^m \\
        -z \lambda^{-m} & z \bar{z} \end{pmatrix}, \hspace{.2cm} 
       \textup{$m$ odd}; \\
\stilde(\psi^m_z) &=& i \sigma_3 + 
   \frac{2i\, m}{1-z\bar{z}} \begin{pmatrix} z \bar{z} &-z \lambda^{-m+1} \\
    \bar{z}\lambda^{m-1} 
         & - z \bar{z} \end{pmatrix}, \hspace{.2cm} \textup{$m$ even}.       
\eeqas 
We know $\psi^m_z = \omega_m \in \mathcal{P}_m$ at $z = 1$, 
and that $\psi^m_z \in \mathcal{B}_{1,1}$ for $|z| \neq 1$;
but, other than the case $m=1$, we see that 
$\stilde(\psi^m_z)$ does not have a finite limit as $z \to 1$.  

We next show that for $\mathcal{P}_2$  this 
behavior is typical. An example corresponding to the following result is 
the two-sheeted hyperboloid of Example \ref{exa:spheres}.
%********************
\begin{theorem} \label{blowupprop2}
Let $\phi_n$ be a sequence in $\mathcal{B}_{1,1}$ with 
$\lim_{n \to \infty} \phi_n = \phi_0 \in \mathcal{P}_2$.
Denote the components of $\stilde(\phi_n)$ by 
\begin{small}$\stilde(\phi_n) = \begin{pmatrix} a_n & b_n \\ 
b^*_n & -a_n \end{pmatrix}$.  
\end{small}
Then 
$\lim_{n\to \infty} |a_n| = \lim_{n\to \infty} |b_n| = \infty$, for 
all $\lambda \in \SSS^1$. 
\end{theorem}
%------------
\begin{proof}
Let $\phi_n = F_n B_n$ be the $SU_{1,1}$ Iwasawa splitting for $\phi_n$,
and $\phi_0 = F_0 \omega_2 B_0$.  Because $\Ad_{\sigma_1} 
\omega_2 = \omega_1$, $\Ad_{\sigma_1} \phi_0 = \Ad_{\sigma_1}F_0 \,
 \Ad_{\sigma_1}\omega_2 \, \Ad_{\sigma_1} B_0$ is in $\mathcal{P}_1$. 
So $\Ad_{\sigma_1} \phi_n $ is a sequence in $\mathcal{B}_{1,1}$ 
which approaches
$\mathcal{P}_1$. Therefore, by Theorem \ref{extensionthm}, 
there exists a finite limit:
\bdm
\lim_{n\to\infty} \sym(\Ad_{\sigma_1} F_n) = L.
\edm
Now
\beq  \label{parallelsym}
\mathcal{S}(\Ad_{\sigma_1} F_n) = \sigma_1[-F_n i \sigma_3 F_n^{-1} +
   2 i \lambda (\partial_\lambda F_n) F_n^{-1}] \sigma_1,
\eeq
and, from Proposition \ref{blowupprop}, we can write
\bdm
F_n \sigma_3 F_n^{-1} = \begin{pmatrix} \pm (|x_n|^2 + |y_n|^2) 
              & -2x_n y_n \\
     2 x^*_n y^*_n & \mp (|x_n|^2 + |y_n|^2) \end{pmatrix}, 
\edm
where $|x_n| \to \infty$, $|y_n| \to 
\infty$.  Thus, all components of the matrix
$F_n i \sigma_3 F_n^{-1}$ blow up as $n \to \infty$, and, for the
limit $L$ to exist it is necessary that all components of the matrix
$\lambda (\partial_\lambda F_n) F_n^{-1}$ also blow up.  
Now we compute
\beqas
\sym (F_n) &=& F_n i\sigma_3 F_n^{-1} + 2i 
\lambda(\partial_\lambda F_n)F_n^{-1} \\
 &=& -[-F_n i\sigma_3 F_n^{-1} + 2i \lambda(\partial_\lambda F_n)F_n^{-1}] +
    4i \lambda(\partial_\lambda F_n)F_n^{-1} \\
    &=& -\sigma_1 \sym(\Ad_{\sigma_1} F_n) 
\sigma_1 + 4i \lambda(\partial_\lambda F_n)F_n^{-1}.
\eeqas 
Since the first term on the right-hand side has the finite limit 
$-\sigma_1 L \sigma_1$, and all components of the second term diverge,
it follows that all components of $\sym (F_n)$ diverge.
\end{proof}

\pagebreak

\begin{corollary} \label{sumcor4}
Proof of item \ref{summary4} of Theorem \ref{summarythm}.
\end{corollary}
\begin{proof}
We just showed that $f^{\lambda_0}$ diverges to $\infty$ as $z \to z_0 \in C_2$.
The  metric is given  on $\Sigma^\circ$ by
the expression $4 \rho^4 (\dd x^2 + \dd y^2)$ (see Remark \ref{metricremark}).
By Proposition \ref{blowupprop}, we have $\rho^4 \to \infty$ as $z \to z_0$.
\end{proof}
%--------------------------------------------------------
%********************************************************

\subsubsection{The higher small cells}
Numerical experimentation shows that the behavior of the surface as $\mathcal{P}_j$
is approached, for $j \geq 3$, may not be so straightforward.
To analyze the behavior analytically becomes more complicated. In principle, one can obtain explicit
factorizations such as in Lemma \ref{switchlemma} by finite linear algebra,
but we do not attempt an exhaustive account here.
One should observe, however, that, relating the Iwasawa decomposition given
here to Theorem (8.7.2) of \cite{PreS} shows that the higher small cells 
occur in higher codimension in the loop group.

\section{Spacelike CMC surfaces of revolution and equivariant surfaces 
in $\R^{2,1}$}\label{section4-wayne}

\subsection{Surfaces with rotational symmetry}
To make general spacelike rotational CMC surfaces in $\R^{2,1}$, 
we convert a result in \cite{SKKR} to the 
$SU_{1,1}$ case.  This theorem provides us with a frame $F$ that gives 
rotationally invariant surfaces when inserted into the Sym-Bobenko formula.  

\begin{theorem}\label{SU11DelaunayIwasawa}
For $a,b \in \mathbb{R}^*$ and $c \in \mathbb{R}$, 
let $\Sigma = \{ z=x+i y \in \C \, | \, 
-\kappa_1^2 < x < \kappa_2^2 \}$ and choose $\kappa_1,\kappa_2$ so 
that $x \in (-\kappa_1^2,\kappa_2^2)$ 
is the largest interval for which a solution $v=v(x)$ of 
\beq \label{vprimeeqn}
\begin{split} 
& (v^\prime)^2 = (v^2-4a^2)(v^2-4b^2)+4c^2v^2 \; , \\
& v^{\prime \prime} = 2 v (v^2-2a^2-2b^2+2c^2), \\
&  v(0) =  2 b, 
\end{split}
\eeq
 is finite and never zero ($\prime$ denotes $\tfrac{d}{dx}$).  
When $c \neq 0$, we require $v^\prime(0)$ and $-b c$ to have the same 
sign.  Let $\phi$ solve $d\phi = \phi \xi$ on $\Sigma$ for 
$\xi = A dz$ with 
\begin{equation}\label{eqn:formforA} 
A = \begin{pmatrix} c & a \lambda^{-1} + 
b \lambda \\ -a \lambda - b \lambda^{-1} & -c \end{pmatrix}
\end{equation} and $\phi(z=0) = I$.  Then we have the 
$SU_{1,1}$ Iwasawa splitting $\phi = F B$, with 
\[ \phi = \exp((x+iy)A)\; , \;\;\; 
   F = \phi \cdot \exp (-f A) \cdot B_1^{-1} \; , \;\;\; 
   B = B_1 \cdot \exp (f A) \; , \] where, 
taking $\sqrt{\det B_0}$ so that $\sqrt{\det B_0}|_{\lambda=0}>0$, 
\beqas
\begin{aligned}
& f = \int_0^x \frac{2 dt}{1 +(4 a b \lambda^2)^{-1} v^2(t)} \; , \\
&B_1 = \frac{1}{\sqrt{\det B_0}} B_0 \; , \hspace{0.7cm} 
B_0 = \begin{pmatrix} 2 v (b+a \lambda^2) & 
     (2 c v+v^\prime) \lambda \\ 0 & 4 a b \lambda^2 + v^2
     \end{pmatrix} \; . 
\end{aligned}
\eeqas
\end{theorem}

The second, overdetermining, equation in 
(\ref{vprimeeqn}) excludes certain enveloping solutions.
In particular it removes constant solutions for $v$, except precisely in the 
case where we want them (when $a=\pm b$ and $c=0$).

%***************************************************************
%###############################################################

\begin{figure}
\begin{center}
\includegraphics[height=.30\linewidth]{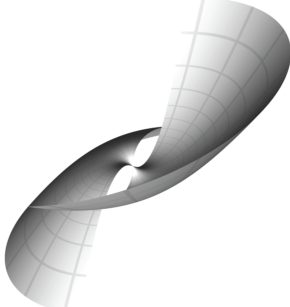} \hspace{.4cm}
\includegraphics[height=.30\linewidth]{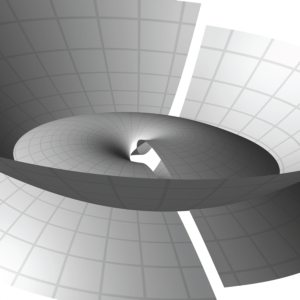} \hspace{.4cm}
\includegraphics[height=.30\linewidth]{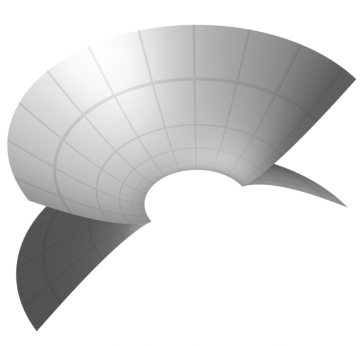}
\end{center}
\caption{A surface of revolution in $T_2$ with timelike axis,
a surface in its associate family, and the parallel constant
Gaussian curvature surface (left to right).
The second and third surfaces
appear to have cuspidal edge singularities.}
\label{fg:4}
\end{figure}

%***************************************************************
%###############################################################

\begin{proof}
Because $B_0|_{z=0} = (4 a b \lambda^2 + 4 b^2) \cdot I$, we have 
$B|_{z=0} = F|_{z=0} = I$.  We set $\Theta=\Theta_1 dx + \Theta_2 dy$, 
where $z=x+iy$, with 
\[ \Theta_1 = \begin{pmatrix}
0 & \tfrac{2 a b}{\lambda v} - \tfrac{v \lambda}{2} \\ 
\tfrac{2 a b \lambda}{v} - \tfrac{v}{2 \lambda} & 0 
\end{pmatrix} \; , \hspace{.7cm}
\Theta_2 = i \begin{pmatrix}
-\tfrac{v^\prime}{2 v} & \tfrac{2 a b}{\lambda v} + \tfrac{v \lambda}{2} \\ 
- \tfrac{2 a b \lambda}{v} - \tfrac{v}{2 \lambda} & \tfrac{v^\prime}{2 v} 
\end{pmatrix} \; . \]  
A computation gives $B_x+(\Theta_1 + i \Theta_2) B = 0$ and 
$\Theta_2 B - i B A = 0$, implying 
$dB+\Theta B - B A (dx+idy) = 0$, and so $F^{-1} dF = \Theta$.  
Noting that 
$\Theta_1 + i \Theta_2$ has no singularity at 
$\lambda = 0$, we have that 
$B$ is holomorphic in $\lambda$ for all $\lambda \in \mathbb{C}$.  
Also, $\text{trace}(\Theta_1 + i \Theta_2) = 0$ 
implies $\det B$ is constant, so 
$\det B = 1$.  Hence
$B$ takes values in $\uhat$. 
We have $\tau(\Theta)= \Theta$, so 
$\tau(F^{-1} dF)=F^{-1} dF$.  It follows from 
$F|_{z=0}=I$ that $\tau(F)=F$, so $F$ takes values in  $\hh_\tau$. 
\end{proof}

\begin{remark}
Note that we must restrict $\kappa_1,\kappa_2$ so that $v$ is 
never zero on $\Sigma$.  
When $v$ reaches zero, this is precisely the moment 
when $\phi$ leaves $\mathcal{B}_{1,1}$.  
Also, note that $v$ can be non-constant even when 
$c = 0$.  A solution to the equation for $v$, for 
example when $0<b<a$ and $c \leq 0$, is given in terms of the 
Jacobi sn function as:
$ v(x) = 2 b \ell^{-1} \text{sn}_{b/(\ell^2 a)}(2 \ell a (x+x_0))$,
where $\ell$ is the largest (in absolute value) of the real solutions
to the equation $a^2\ell^4 + (c^2-a^2-b^2)\ell ^2 + b^2 =0$,
 and $x_0$ is chosen so that $v(0)=2b$ and $v^\prime(0) 
\geq 0$.  
\end{remark}

Inserting the above $F$ into the Sym-Bobenko formula, we get 
explicit parametrizations of CMC rotational surfaces in $\R^{2,1}$.  
Because the mean curvature $H$ and the Hopf differential 
term $Q$ are constant reals, 
and because the metric $ds^2$ is invariant under 
translation of the $z$-plane in the 
direction of the imaginary axis, we conclude that these surfaces are 
rotationally symmetric, by the fundamental theorem 
for surface theory, and we have the following corollary.  

\begin{corollary}\label{lastresult}
Inserting $F$ as in Theorem \ref{SU11DelaunayIwasawa} into 
\eqref{symformula} with $\lambda_0=1$, we have a surface of revolution 
$\hat f^1$ with axis parallel to the line through $0$ and $i A$ in 
$\R^{2,1} \approx \mathfrak{su}_{1,1}$.  
In particular, the axis is timelike, null or spacelike 
when $(a+b)^2-c^2$ is negative, zero or positive, respectively.
\end{corollary}

\begin{proof}
The rotational symmetry of the surface is represented by 
$F \to \exp(i y_0 A) F$ at $\lambda_0=1$ for each $y_0 \in 
\R$, and the Sym-Bobenko formula 
changes from $\hat f^1$ to \[ 
\exp(i y_0 A) \hat f^1 \exp(-i y_0 A) - 
i H^{-1} \partial_\lambda (\exp(i y_0 A))|_{\lambda=1} \cdot 
\exp(-i y_0 A) \; . \]  
The axis is then a line parallel to the line invariant under 
conjugation by $\exp(i y_0 A)$.  
\end{proof}

\begin{remark}
Using conjugation by $\text{diag}(\sqrt{i},1/\sqrt{i})$ on all of 
$A$, $\phi$, $F$, $B$, one can see that if we choose $H=-2 a b$, 
Equation \eqref{withlambda} 
gives $v=e^{-u}$ and $Q=1$, for the surfaces in 
Corollary \ref{lastresult}.  
\end{remark}

\subsection{Equivariant surfaces}
By inserting the $F$ in Theorem \ref{SU11DelaunayIwasawa} 
into \eqref{symformula} and evaluating at various 
values of $\lambda_0 \in \mathbb{S}^1$, we get surfaces in the associate 
families of the surfaces of revolution in 
Corollary \ref{lastresult}.  These are the equivariant 
surfaces, which we now describe.  

\begin{definition}
An immersion $f: U \subset \R^2\to\MINK$
is equivariant with respect to $y$
if there exists a continuous homomorphism $R_t:\R\to\calE$
into the group $\calE$ of isometries of $\MINK$
such that
\begin{equation*}
\label{eq:equivariant-def}
f(x,\,y+t) = R_t f(x,\,y).
\end{equation*}
\end{definition}

In the following we write $z=x+iy$.

%****************************************************************
%---------------------------------------------------
% equivariant immersions
%---------------------------------------------------
\begin{proposition}
\label{prop:equivariant1}
Let $f: U \subset \R^2\to\MINK$ be a conformal immersion
with metric $4 v^{-2} \abs{dz}^2$, 
mean curvature $H$, and Hopf differential $Q\,dz^2$.
Then $f$ is equivariant with respect to $y$ if and only if
$v$, $H$ and $Q$ are $y$-independent.
\end{proposition}

\begin{proof}
The proposition is 
shown by the following sequence of equivalent statements:

1. The immersion $f$ is equivariant with respect to $y$.

2. For any $t\in\R$,
the maps $f(x,\,y)$ and $f_t(x,\,y) = f(x,\,y+t)$
differ by an isometry $R_t$ of $\MINK$.
To show statement 1 from 2, note that
$R_{s+t}f(x,\,y) = f(x,\,y+s+t) = R_t f(x,\,y+s) = R_t R_s f(x,\,y)$, 
so under suitable non-degeneracy conditions on $f$, the map $t\mapsto R_t$ is
a continuous homomorphism.

3. The immersions $f$ and $f_t$ have equal first and second
fundamental forms. Statements 2 and 3 are equivalent by the
fundamental theorem of surface theory.

4. The geometric data for $f$ satisfy
$v(x,\,y+t) = v(x,\,y)$, $H(x,\,y+t)=H(x,y)$ and $Q(x,\,y+t) = Q(x,\,y)$.

5. The functions $v$, $H$ and $Q$ are $y$-independent.
\end{proof}

%---------------------------------------------------
% equivariant CMC immersions
%---------------------------------------------------

\begin{proposition}
\label{prop:equivariant2}
Let $f: \Sigma \subset \C \to\MINK$ be a conformal CMC $H$ immersion
with metric $4 v^{-2} \abs{dz}^2$
and Hopf differential  $Q\,dz^2$.
Take $q \in \real^* := \real \setminus \{ 0 \}$ so that $4H^2 = q^2$, 
and suppose $|Q|$ is $1$ at some point in $\R^2$.  
Then $f$ is equivariant with respect to $y$
if and only if
$Q$ is constant,
$v$ depends only on $x$, and for some $p\in\R$, $v$ satisfies
\begin{equation}
\label{eq:equivariant-v}
{v'}^2 = v^4 -2 p v^2 + q^2 , \quad
v'' = 2v(v^2-p).
\end{equation}
\end{proposition}

\begin{proof}
If $f$ is equivariant, then $v$ and $Q$ are $y$-independent
by Proposition~\ref{prop:equivariant1}.
Since $f$ is CMC, then $Q$ is holomorphic in $z$, and is hence constant.
So $|Q| \equiv 1$.  
Since $v$ is $y$-independent, the Gauss equation~\eqref{compatibility1}
with $v=e^{-u}$ is 
a second order ODE in $x$.  Multiplying the Gauss equation
by $u'$ and integrating yields~\eqref{eq:equivariant-v}, where
$p$ is a constant of integration.

Conversely, if $v$ and $Q$ satisfy the conditions of the proposition,
then $f$ is equivariant, by Proposition~\ref{prop:equivariant1}.
\end{proof}

%---------------------------------------------------
% all equivariant CMC immersions
%---------------------------------------------------
\begin{corollary}
\label{cor:equivariant3}
Any immersion $\hat f^{\lambda_0}$ into $\MINK$ as in 
\eqref{symformula}, obtained from 
a DPW potential of the form $Adz$,
where $A$ is given by~\eqref{eqn:formforA}, 
is a conformal CMC immersion equivariant with respect to $y$.

Conversely, up to an isometry of $\MINK$, every non-totally-umbilic
conformal  spacelike CMC $H \neq 0$ 
immersion equivariant with respect to $y$ is obtained from
some DPW potential $Adz$, where $A$ is of the form~\eqref{eqn:formforA}.
\end{corollary}

\begin{proof}
By Theorem~\ref{SU11DelaunayIwasawa},
the extended frame of the immersion obtained
from $A$ is of the form $F(x,\,y) = \exp(iyA)\mathcal{G}(x)$ for 
some map
$\mathcal{G}: J \to SU_{1,1}$, where $J = (-\kappa_1^2, \kappa_2^2) \subset \real$.
   The Sym-Bobenko formula 
$\hat f^{\lambda_0}$ applied to $F$
yields an immersion which is equivariant with respect to $y$.

Conversely, given a CMC immersion
 $f: (-\tilde{\kappa}_1^2, \tilde{\kappa}_2^2) \times \real \subset  \R^2\to\MINK$, which is equivariant
with respect to the second coordinate $y$, let $4 v^{-2} \abs{dz}^2$ and $Qdz^2$
be the metric and Hopf differential 
 for $f$, respectively.
By a dilation  of coordinates $z \to r z$ for a constant 
$r \in  \R$, we may assume $|Q|=1$.
Let $q$ be as in Proposition~\ref{prop:equivariant2}.
  By that proposition,
$v$ satisfies~\eqref{eq:equivariant-v} for some $p\in\R$.
Let $b = v(0)/2$, and define $a \in \real^*$ by the equation $H=-2 a b$,
and so $q=\pm 4ab$.  
Then it follows that $p \leq 2 (a^2+b^2)$, and 
there exist $c\in\R$ and $\lambda_0 \in \SSS^1$
such that $p = 2(a^2+b^2-c^2)$ and $Q = \lambda_0^{-2}$.  
Let $\hat f^{\lambda_0}$ be the immersion induced
from the DPW potential $\xi =A\dd z$, with $A$ as in
 Theorem \ref{SU11DelaunayIwasawa}, initial condition
$\Phi(0)=\id$, and $\lambda_0$ and $H=-2 a b$.  Then 
$\hat f^{\lambda_0}$ has metric $4 v^{-2} \abs{dz}^2$, 
by Theorems \ref{SU11DelaunayIwasawa} and \ref{dpwthm}, and has
mean curvature $-2 a b$ and Hopf differential
$\lambda_0^{-2} dz^2$. By the fundamental theorem of surface theory,
$f$ and $\hat f^{\lambda_0}$ differ by an isometry of $\MINK$.
\end{proof}

%---------------------------------------------------
% moduli space
%---------------------------------------------------
We now describe the two spaces $R/\!\!\sim_R$ and $E/\!\!\sim_E$
of immersions into $\MINK$ which are rotationally invariant and
equivariant, respectively.
Both constructions are based on the family
of solutions to the integrated Gauss equation~\eqref{eq:equivariant-v},
where solutions are identified which amount to a coordinate shift
and hence yield ambiently isometric immersions.
Bifurcations in the space of solutions to 
Equation~\eqref{eq:equivariant-v}
lead to non-Hausdorff 
quotient spaces.

%------------------------------------------------------------
% begin moduli space image
%------------------------------------------------------------

\begin{figure}[ht]
\includegraphics[width=5.4in]{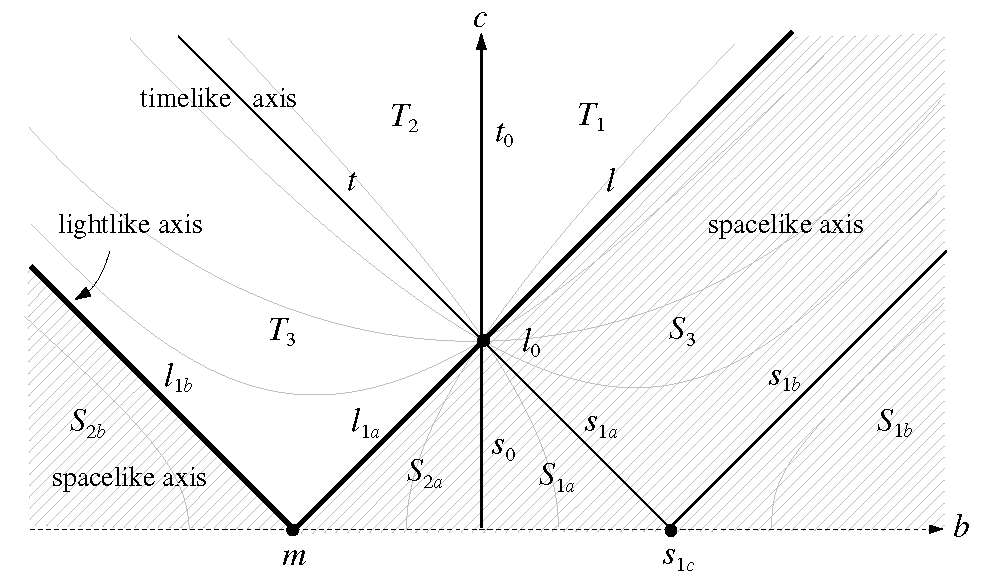}
\caption{\small A blowup of the moduli space of surfaces with rotational
symmetry $\R^{2,1}$.
The blowdown is obtained by identifying points along segments
of hyperbolas within each region.
The heavily drawn left v-shaped line represents the lightlike
axis examples, separating those with spacelike and timelike axes.
The line segments and points in the diagram represent examples
whose metrics degenerate to elementary functions;
in particular, the $c$-axis represents hyperboloids.
Pairs in the same associate family are represented by points
reflected across the $c$-axis.
}\label{fig:moduli}
\end{figure}

%------------------------------------------------------------
% end moduli space image
%------------------------------------------------------------

The space $R/\!\!\sim_R$ of immersions with rotational symmetry is a quotient
of the space
\begin{equation*}
\label{eq:equivariantR}
R = \{(p,\,q,\,v_0)\in\R^3 \suchthat v_0^4-2pv_0^2+q^2 \ge 0\}
\end{equation*}
parametrizing solutions to~\eqref{eq:equivariant-v}.
The equivalence relation $\sim_R$ on $R$ is defined as follows:
$(p_1,\,q_1,\,v_1)\!\!\sim_R\!\!(p_2,\,q_2,\,v_2)$ if,
for $k=1$ and $2$, the respective solutions to
\begin{equation} \label{vprimesquaredeqn}
\begin{split}
&{v'}^2 = v^4 - 2 p_k v^2 + q_k^2, \\
& v'' =  2v(v^2-p_k),\\
& v(0) = v_k,
\end{split}
\end{equation}
are equivalent in the following sense:
there exist $r\in\R_+$ and $c\in\R$ such that $v_2(x) = rv_1(rx+c)$
or $v_2(x) = -rv_1(rx+c)$.
The space $R/\!\!\sim_R$ is a $1$-dimensional 
non-Hausdorff manifold. For a point in $R$ with 
$q\neq 0$, the corresponding surface
is constructed by relating (\ref{vprimesquaredeqn}) to (\ref{vprimeeqn}).
This determines $a$, $b$ and $c$ in Theorem \ref{SU11DelaunayIwasawa},
and the surface is given by Corollary \ref{lastresult}. If $q=0$, 
the surface is totally umbilic.

To describe the space of equivariant immersions, let
\begin{equation*}
\label{eq:equivariantE}
E = \{(p,\,P,\,v_0)\in\R\times\C\times\R
\suchthat v_0^4-2pv_0^2+\abs{P}^2 \ge 0\}.
\end{equation*}
The equivalence relation $\sim_E$ on $E$ is
defined as follows: $(p,\,P,\,v_0)\!\!\sim_E\!\!(p',\,P',\,v_0')$ if
there exist $q,\,q'\in\R$ and $\lambda\in\SSS^1$
such that $P = q\lambda^{-2}$ and $P' = q'\lambda^{-2}$, and
$(p,\,q,\,v_0)\!\!\sim_R\!\!(p',\,q',\,v_0')$.
The space $E/\!\!\sim_E$ is a $2$-dimensional 
non-Hausdorff manifold.  The surface corresponding to a point in 
$E$, when $P \neq 0$, is as  in the case of the space $R$, 
except that the Sym-Bobenko formula now uses general $\lambda \in \SSS^1$ (not 
necessarily $\lambda = 1$).  When $P=0$, the surface is totally umbilic. 

The above constructions are summarized as: 
\begin{theorem}
\label{thm:equivariant-moduli}
Up to coordinate change and ambient isometry,
the spaces $E/\!\!\sim_E$ and $R/\!\!\sim_R$ are the moduli spaces of
CMC immersions into $\MINK$ which are respectively 
equivariant and rotational.
\end{theorem}

%=========================================================================
%================================================
% BEGIN NEW TEXT
%================================================

\subsection{The moduli space of surfaces with rotational symmetry}

Figure~\ref{fig:moduli} shows a blowup of the moduli space of surfaces with
rotational symmetry in $\MINK$.
The underlying space
is the closed $(b,c)$-half-plane
obtained by the normalization $\lambda=1$ and $a=1$.
The blowdown to the 1-dimensional moduli space of surfaces
is the quotient modulo identification of points on segments of hyperbolas
$1+b^2-c^2 = (\mathrm{constant})\cdot b$
 foliating each region.
The examples with spacelike, timelike and lightlike axes
are represented respectively by the shaded and unshaded regions,
and the left heavily-drawn v-shaped line.
Subscripted letters $S$, $L$ and $T$ denote
one-parameter families with spacelike, lightlike and timelike axes, 
respectively;
likewise, $s$, $\ell$ and $t$ designate single examples,
and the example $m$ has no axis.
The moduli space is a connected non-Hausdorff space, and  
 is the disjoint union of eight one-parameter families
$S_{1a}$, $S_{1b}$, $S_{2a}$, $S_{2b}$, $S_3$,
 $T_1$, $T_2$, $T_3$,
eight individual examples
$s_{1a}$, $s_{1b}$, $s_{1c}$,
$\ell_{1a}$, $\ell_{1b}$, $m$, $\ell$, $t$,
and the hyperboloids corresponding to
$s_0$, $\ell_0$, $t_0$ considered with
spacelike, lightlike and timelike axes respectively.

The non-Hausdorffness of the moduli space arises from the fact that
the limit surface of a sequence of surfaces in any of the one-parameter
families (designated by capital letters)
to a point not in that family is not
uniquely determined:
the sequence will have different limit surfaces
depending on how the sequence is chosen to be positioned in $\MINK$.
The blowup of the moduli space shown in Figure~\ref{fig:moduli}
maps this topology.
For example, the same sequence of surfaces
in $S_3$ can converge to either $s_{1a}$, $s_{1b}$ or $s_{1c}$;
likewise a sequence in $T_3$ can converge to 
either $l_{1a}$, $l_{1b}$ or $m$.\\

%------------------------------------------------------------
% begin imageset 1
%------------------------------------------------------------
\begin{figure}[ht]
\centering
$
\begin{array}{cccc}
\includegraphics[width=3cm]{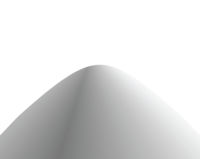}
& \includegraphics[width=3cm]{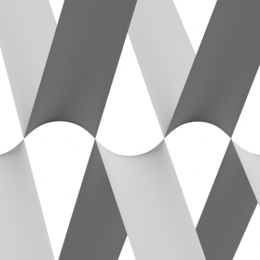}
& \includegraphics[width=3cm]{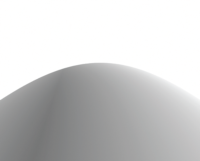}
& \includegraphics[width=3cm]{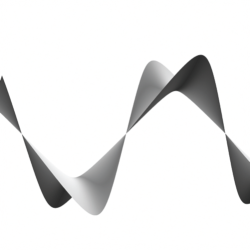}
\\
  \delcaption{S_{1a}}{\frac{1}{2}}{0}
  & \delcaption{S_{1b}}{2}{0}
  & \delcaption{S_{2a}}{-\frac{1}{2}}{0}
  & \delcaption{S_{2b}}{-2}{0}
\vspace{2ex} \\
\includegraphics[width=3cm]{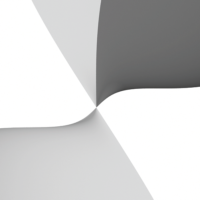}
& \includegraphics[width=3cm]{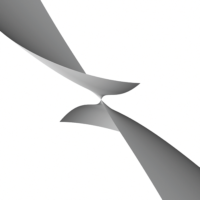}
& \includegraphics[width=3cm]{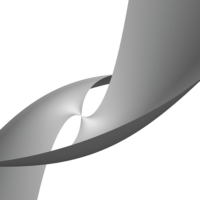}
& \includegraphics[width=3cm]{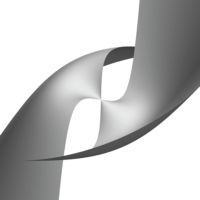}
\\
  \delcaption{S_3}{1}{\sqrt{2}}
  & \delcaption{T_1}{1}{4}
  & \delcaption{T_2}{-1}{4}
  & \delcaption{T_3}{-1}{\sqrt{2}}
\end{array}
$
\caption{Examples from each of the eight families of
surfaces with rotational symmetry in $\MINK$.
These families together with the eight single examples shown
in Figure~\ref{fig:delaunay2} comprise all such surfaces.
The designation symbol and the numbers $(b,\,c)$
refer to the blowup of the moduli space in Figure~\ref{fig:moduli}.
Note that entire examples are necessarily complete \cite{cheng-yau}.
Images created with XLab \cite{xlab}.} 
\end{figure}

%------------------------------------------------------------
% end imageset 1
%------------------------------------------------------------
%------------------------------------------------------------
% begin imageset 2
%------------------------------------------------------------

\begin{figure}[ht]
\centering
$
\begin{array}{cccc}
\includegraphics[width=3cm]{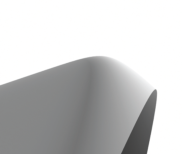}
& \includegraphics[width=3cm]{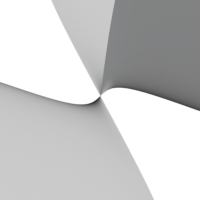}
& \includegraphics[width=3cm]{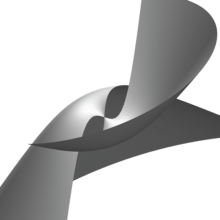}
& \includegraphics[width=3cm]{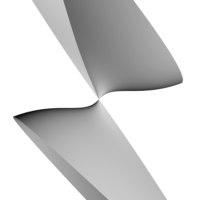}
\\
  \delcaption{s_{1a}}{\frac{1}{2}}{\frac{1}{2}}
  & \delcaption{s_{1b}}{2}{1}
  & \delcaption{t}{-1}{2}
  & \delcaption{\ell}{1}{2} \\
%\vspace{1ex} \\
\includegraphics[width=3cm]{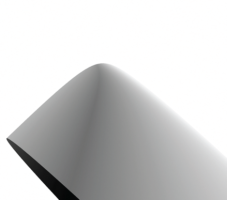}
& \includegraphics[width=3cm]{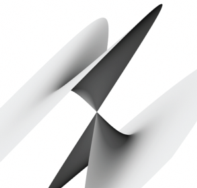}
& \includegraphics[width=3cm]{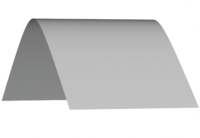}
& \includegraphics[width=3cm]{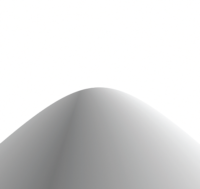}
\\
 \delcaption{\ell_{1a}}{-\frac{1}{2}}{\frac{1}{2}}
 & \delcaption{\ell_{1b}}{-2}{1}
 & s_{1c},\,m
 & s_0,\,\ell_0,\,t_0
\vspace{2ex}
\end{array}
$
\caption{The eight surfaces with rotational symmetry in $\MINK$ whose
metric is an elementary function.
The last two examples,
a cylinder over a hyperbola and a hyperboloid
respectively, appear multiple times in the blowup of the moduli space.
Designation symbols 
are as  in Figure~\ref{fig:moduli}.
Images created with XLab~\cite{xlab}.} \label{fig:delaunay2}
\end{figure}

%------------------------------------------------------------
% end imageset 2
%------------------------------------------------------------

%================================================
% END NEW TEXT
%================================================

%************************************************************

\section{Analogues of Smyth surfaces in $\mathbb{R}^{2,1}$}
\label{section5-wayne}
A generalization of Delaunay surfaces in $\R^3$ was studied by B. Smyth in \cite{smyth}.
These are constant mean curvature surfaces whose metrics are invariant under rotations.
They were also studied by Timmreck et al. in \cite{FerPT}, where they were shown
to be properly immersed (a property which we will see does not hold for the 
analogue in $\real^{2,1}$).  The DPW approach was applied in \cite{DorPW} and
\cite{bi}.

Here we use the DPW method to construct the analogue of Smyth  surfaces 
 in $\R^{2,1}$, and describe some of their properties.  Define 
\begin{equation}\label{smyth-potential} 
\xi = \lambda^{-1} \bbar 0 & 1\\  c z^k  & 0 \ebar dz \; , \;\;\; c \in 
\mathbb{C} \; , \;\;\; z \in \Sigma = \C \; , \end{equation} 
and take the solution $\phi$ of $d\phi = \phi \xi$ with 
$\phi |_{z=0}=I$.  If $k=0$ and $c \in \mathbb{S}^1$, then 
one can explicitly split $\phi$ as in Example 
\ref{exa:cylinders}, and the 
resulting CMC surface is a cylinder over a hyperbola whose axis depends on the choice 
of $c$.  When $c=0$, one produces a two-sheeted hyperboloid.  
However, when $c \not\in \mathbb{S}^1 \cup \{ 0 \}$ or when $k>0$, 
Iwasawa splitting of $\phi$ is not so simple.  

Changing 
$c$ to $c e^{i \theta_0}$ for any $\theta_0 \in \mathbb{R}$ only 
changes the resulting surface 
by a rigid motion and a reparametrization $z\to
ze^{-\frac{i\theta_0}{k+2}}$.  
So without loss of generality we assume that 
$c \in \mathbb{R}^+ := \real \cap (0,\infty)$.  

\begin{lemma}\label{lm:smyth}
The surfaces $f: \Sigma^\circ = \phi^{-1}(\mathcal{B}_{1,1}) \to \real^{2,1}$,
produced via the DPW method, from $\xi$ 
in \eqref{smyth-potential}, with $\phi|_{z=0}=\id$ and 
$\lambda_0=1$, have reflective symmetry with respect to $k+2$ geodesic 
planes that meet equiangularly along a geodesic line.  
\end{lemma}

\begin{proof}
Consider the reflections 
\bdm
R_\ell(z)=e^{\frac{2 \pi i \ell}{k+2}} \bar z,
\edm
 of the domain 
$\Sigma = \mathbb{C}$, for $\ell \in \{0,1,...,k+1\}$. In the coordinate
$w := R_\ell (z)$, we have 
\bdm
\xi = A_\ell \Big ( \lambda^{-1} \bbar 0 & 1 \\ c \bar{w}^k & 0 \ebar \dd \bar{w} \Big ) A_\ell^{-1}, 
   \hspace{1cm} A_\ell = \text{diag}(e^{\frac{\pi i \ell}{k+2}},
e^{\frac{-\pi i \ell}{k+2}}).
\edm
Comparing this with (\ref{smyth-potential}), it follows that 
$\phi (z) = A_\ell \phi(\bar{w}) A_\ell^{-1}$, and hence 
\bdm
\phi (R_\ell (z)) = A_\ell \phi(\bar{z}) A_\ell^{-1}.
\edm
It is easy to see that this relation extends to the factors $F$ and $B$ in
the Iwasawa splitting $\phi = FB$, and so we have a frame $F$ which satisfies
\bdm
F(R_\ell(z)) = A_\ell F(\bar{z}) A_\ell^{-1}.
\edm
Since we have assumed $c \in \real^+$, it follows from the form of $\xi$ and
the initial condition for $\phi$ that
$\phi(\bar{z}, \lambda) = \overline{\phi(z, \bar{\lambda})}$.
This symmetry also extends to the factors $F$ and $B$, and combines with
the first symmetry as: $F(R_\ell(z), \lambda) = A_\ell \overline{F(z, \bar{\lambda})}A_\ell^{-1}$.   
Inserting this into \eqref{symformula}, we have
\[ \hat f^\lambda(R_\ell(z)) 
 = -A_\ell \overline{\hat f^{\bar \lambda}(z)} A_\ell^{-1} \; . \] 
Then for $\hat f^1$, 
the transformation $\hat f^1 \to -\overline{\hat f^1}$ 
represents reflection across the plane $\{x_2=0\}$ of 
$\mathbb{R}^{2,1} = \{ x_1e_1+x_2e_2+x_0e_3 \}$, 
and conjugation by $A_\ell$ represents a rotation by angle $2\pi i\ell /(k+2)$ 
about the $x_0$-axis. 
\end{proof}

%**************************************************************

\begin{figure}[ht]
\label{fig:smyth}
\centering
\includegraphics[height=4cm]{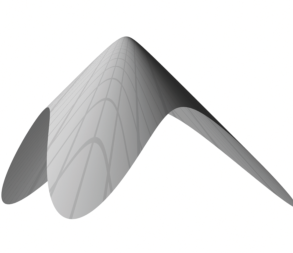} \hspace{.2cm}
\includegraphics[height=4cm]{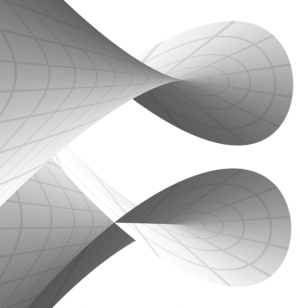} \hspace{.2cm}
\includegraphics[height=4cm]{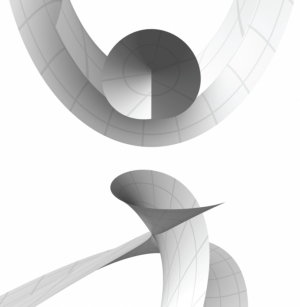}

\caption{Details of Smyth surface analogs in $\R^{2,1}$.
An immersed portion of this surface (left) has three-fold ambient
rotational symmetry, Lemma~\ref{lm:smyth}.
A singularity further out on the surface appears to be a
swallowtail (second image). The third image shows another singularity on the same
surface.}
\label{fg:1}
\end{figure}
%****************************************************************
%################################################################

We now show that $u: \Sigma^\circ  \to \mathbb{R}$ 
in the metric \eqref{eqn:dssquared} 
of the surface resulting from the frame $F$ 
is constant on each circle of radius $r$ centered at 
the origin in $\Sigma$, that is, 
$u=u(r)$ is independent of $\theta$ in $z=r e^{i \theta}$.  
Having this internal rotational symmetry of the metric (without actually 
having a surface of revolution) is what defines the surface as 
an analogue of a Smyth surface.  

\begin{proposition}
\label{prop:Smyth rotational symmetry}
The solution $u$ of the Gauss equation 
\eqref{compatibility1} for a surface generated 
by $\xi$ in \eqref{smyth-potential}, with 
$\phi |_{z=0}=I$, is rotationally symmetric.  That 
is, $u$ depends only on $|z|$.
\end{proposition} 

\begin{proof}
Define 
\bdm \tilde z =e^{i\theta}z,  \hspace{1cm} \tilde{\lambda}=
e^{i\theta}e^{\frac{i\theta k}{2}}\lambda,
\edm 
for any fixed $\theta \in{\mathbb{R}}$.  Then 
\bdm
\xi(z,\lambda) = L^{-1} \, \Big (  \tilde{\lambda}^{-1}  \bbar 0 & 1 \\ c \tilde{z}^k & 0 \ebar \dd \tilde{z} \Big )  \, L, \hspace{1cm} L = \bbar e^{\frac{-i k \theta}{4}} & 0 \\
    0 & e^{\frac{i k \theta}{4}} \ebar.
\edm  
It follows that
\bdm
\phi (\tilde z,\tilde{\lambda})=L\,\phi(z,\lambda)\,L^{-1}.
\edm  
Let $\phi = F B$ be the normalized Iwasawa splitting of $\phi$, with
$B: \Sigma^\circ \to \uhat$.  Then 
\beqas
\phi (\tilde z,\tilde{\lambda})&=&(L F(z,\lambda) L^{-1})\cdot(L B(z,\lambda) L^{-1})\\
 &=& F(\tilde z, \tilde \lambda) \cdot B(\tilde z, \tilde \lambda).
\eeqas
Since $L B(z,\lambda) L^{-1}$ and $B(\tilde z, \tilde \lambda)$ are both loops
in $\uhat$, and the  left factors are both loops in $\hh$, it follows by 
uniqueness that
the corresponding factors are equal.  Recall from Section \ref{section3-wayne} 
that $u = 2 \log \rho$ is determined by the function $\rho(z)$, which is the
first component of the diagonal matrix $B(z) \big|_{\lambda = 0}$. Since this
matrix is diagonal and independent of $\lambda$, we have
just shown that $B(z) \big|_{\lambda = 0} = B(\tilde z) \big|_{\lambda = 0}$,
and hence $u(\tilde z) = u(z)$. 
\end{proof}
 
We now show that the Gauss equation for these surfaces 
in $\R^{2,1}$ is a special 
case of the Painleve III equation.  This was proven for 
Smyth surfaces in $\R^3$, in \cite{bi}. 

\begin{proposition}
\label{prop:Painleve III}
The Gauss equation \eqref{compatibility1} for a 
surface generated by $\xi$ in \eqref{smyth-potential}, with 
$\phi |_{z=0}=I$, 
is a special case of the Painleve III equation.
\end{proposition}

\begin{proof}
The Painleve III equation, for constants 
$\alpha, \beta, \gamma, \delta$, is 
\bdm
y^{\prime\prime}=y^{-1}(y^{\prime})^2-x^{-1}y^\prime + 
    x^{-1}(\alpha y^2+\beta)+\gamma y^3 + \delta y^{-1},
 \edm  
 where $\prime$ denotes the derivative with respect to $x$.
Setting 
$y=e^v$, $\alpha=\beta=0$, $\gamma=-\delta=1$, 
we have 
$(v^{\prime}e^v)^\prime=e^{-v} {(v^{\prime}e^v)}^2 -
x^{-1} v^{\prime}e^v +
 0+e^{3v}-e^{-v}$.  
Therefore
\begin{equation} 
 v^{\prime \prime}+x^{-1} 
 v^{\prime} - 2\sinh (2v)=0 \label{eq:Painleve III}
 \end{equation}
is a particular case of the Painleve III equation.
 
By a homothety 
and/or reflection, we may assume 
the surface has $H = 1/2$, and then we 
have $Q=-cz^k$.  (By Section 
\ref{section3-wayne}, $Q= - 2 H b_{-1}/a_{-1}$.)  
Setting $r:= |z|$, the Gauss equation becomes 
\begin{equation}
\label{eq:Gauss of Smyth}
4u_{z\bar z}+c^2 \,r^{2k} e^{-2u}-e^{2u}=0\; .
\end{equation}

To prove this proposition,\ we show that Equation 
\eqref{eq:Gauss of Smyth} can be written 
in the form \eqref{eq:Painleve III}.  Set 
\[ v:=u-\tfrac{1}{2}\log |Q|=u-\tfrac{1}{2}\log c-
\tfrac{k}{2}\log r \; , \]
so
$4v_{z\bar z}+\tfrac{k}{2}(\log r)_{z\bar z}+
c^2\,r^{2k}\,
e^{-2(v+\frac{1}{2}\log c+\frac{k}{2}\log {r})}-
e^{2(v+\frac{1}{2}\log c+
\frac{k}{2}\log {r})}=0$, and this simplifies to
\bdm
4v_{z\bar z}-2c\,r^k\,\sinh (2v)=0.
\edm  
Now $v$ is a function of $r$ only, which means that 
$v_{z \bar{z}} = \frac{1}{4}(v^{\prime \prime}(r) + \frac{1}{r} v^\prime(r))$, 
and the equation becomes
\bdm
v^{\prime \prime}(r) +r^{-1} v^\prime(r) - 
2c\,r^k\,\sinh (2v)=0.
\edm  Now set
\bdm
\mu:=(1+\frac{k}{2})^{-1} r^{1+\frac{k}{2}} \,\sqrt{c}.
\edm  
Then
$\partial_{r}\mu=r^{\frac{k}{2}}\,\sqrt{c}$.  
So we have 
\[ \partial_{r}(\partial_\mu v \,r^{\frac{k}{2}}\sqrt{c})+
r^{-1}\partial_\mu v 
\,r^\frac{k}{2}\sqrt{c}-2c\,r^k\,\sinh (2v)
=0\; , \] which simplifies to
$v_{\mu\mu}+\mu^{-1} \,v_\mu - 2\sinh (2v)=0$, 
coinciding with  \eqref{eq:Painleve III}.
\end{proof}

\providecommand{\bysame}{\leavevmode\hbox to3em{\hrulefill}\thinspace}
\providecommand{\MR}{\relax\ifhmode\unskip\space\fi MR }
%% \MRhref is called by the amsart/book/proc definition of \MR.
\providecommand{\MRhref}[2]{%
  \href{http://www.ams.org/mathscinet-getitem?mr=#1}{#2}
}
\providecommand{\href}[2]{#2}

\end{document}